\documentclass[a4paper,12pt]{article}

\usepackage{amssymb,enumitem}
\usepackage[title]{appendix}
\usepackage{tikz-cd}
\usepackage{tikz}
\usetikzlibrary{arrows.meta}
\usepackage{amsfonts,amsmath,amscd,amsthm,color,comment,cases}
\usepackage[all]{xy}
\usepackage[normalem]{ulem}
\usepackage[colorlinks]{hyperref}
\usepackage{verbatim}
\usepackage{graphicx} 
\usepackage{bm}
\usepackage{geometry}
\usepackage{hyperref}
\usepackage{xcolor}
\usepackage{booktabs}

\newfont{\eufm}{eufm10 scaled\magstep1}

\newcommand{\cC}{\mathcal{C}}
\newcommand{\cD}{\mathcal{D}}
\newcommand{\cE}{\mathcal{E}}

\newcommand{\cL}{\mathcal{L}}

\newcommand{\cR}{\mathcal{R}}
\newcommand{\cS}{\mathcal{S}}

\newcommand{\cN}{\mathcal{N}}
\newcommand{\cZ}{\mathcal{Z}}

\newcommand{\bbN}{\mathbb{N}}
\newcommand{\bbZ}{\mathbb{Z}}
\newcommand{\bbC}{\mathbb{C}}

\newcommand{\bbR}{\mathbb{R}}

\newcommand{\cn}{\textbf{cn}}
\newcommand{\tw}{\textrm{wr }}

\newcommand{\K}{\mathbf{K}}
\newcommand{\Spec}{\textrm{Spec}}
\newcommand{\coC}{\mathbf{C}}

\def\dres{\partial{\rm Res}}

\def\ord{\rm ord}

\def\Ker{\rm Ker}
\def\KdV{\rm KdV}
\def\kdv{\rm kdv}

\def\gcrd{\rm gcrd}

\def\ord{\rm ord}
\def\BC{\texttt{BC}}
\def\Ker{{\rm Ker}}

\def\KdV{\rm KdV}

\def\eM{e_{M}}
\def\D{\texttt{D}}
\newcommand{\dx}{\partial_x}
\newcommand{\dxx}{\partial_{xx}}
\newcommand{\dxxx}{\partial_{xxx}}
\newcommand{\dz}{\partial_z}

\newcommand{\fdiff}[2]{\dfrac{\delta #1}{\delta #2}}
\usepackage{mathtools}
\def\para{\vspace{1.5 mm}}

\newtheorem{thm}{Theorem}[section]
\newtheorem{lem}[thm]{Lemma}
\newtheorem{cor}[thm]{Corollary}
\newtheorem{prop}[thm]{Proposition}

\newtheorem{rem}[thm]{Remark}

\newtheorem{conjecture}{Conjecture}[section]

\title{From Spectral Curves to Variational Equations: Geometry and Differential Galois Theory of KdV Cnoidal Waves}

\author{Juan J. Morales-Ruiz, Jean Pierre Ramis, Maria-Angeles Zurro}

\date{}

\begin{document}

\maketitle
\begin{abstract}
   We provide a spectral-geometric description of the KdV variational equation around cnoidal waves in terms of two interacting spectral curves. The second symmetric power of the Lamé–Schr\"odinger operator links squared eigenfunctions, KdV conserved densities and variational solutions, while the spectral curve of the third-order variational operator yields explicit closed-form solutions. The branch points of the spectral curve appear as distinguished loci where these constructions become unified, revealing new connections between spectral geometry and differential Galois theory.
\end{abstract}
\tableofcontents

\section{Introduction}\label{sec-intro}

The interplay between integrable nonlinear equations, algebraic geometry and spectral theory has been one of the central themes in mathematical physics since the seminal works of Burchnall and Chaundy, Novikov, Dubrovin, Krichever and their collaborators. The discovery that commuting ordinary differential operators are encoded by algebraic curves led to the modern algebro-geometric approach to integrable systems, where finite-gap solutions of nonlinear evolution equations are described by spectral curves, Baker–Akhiezer functions and their associated geometric structures.

Among the most prominent examples of this paradigm stands the Korteweg–de Vries (KdV) equation, whose finite-gap solutions are naturally associated with algebro-geometric Schrödinger operators and their spectral curves. In this framework, periodic and quasi-periodic waves arise from algebro-geometric data, and the spectral analysis of the corresponding Schrödinger operator plays a fundamental role in the dynamics of the nonlinear equation. The simplest finite-gap situation is represented by the cnoidal waves, which are associated with Lamé-type Schrödinger operators and elliptic spectral curves.

Parallel to these developments, Differential Galois Theory has provided a powerful algebraic approach to the study of linear differential equations. In particular, Picard–Vessiot theory has proven to be a natural setting for understanding the algebraic structure of integrable systems and their associated linear problems. More recently, the notion of spectral Picard–Vessiot extension has been introduced for algebro-geometric Schrödinger operators, revealing a direct connection between the geometry of the spectral curve and the differential-Galois structure of the corresponding eigenvalue problem. From this viewpoint, the spectral curve not only encodes commuting operators but also governs the construction of eigenfunction bundles and their associated differential field extensions.

Despite the extensive literature devoted to finite-gap Schrödinger operators and KdV solutions, the spectral-geometric structure of the variational equations around finite-gap backgrounds remains comparatively less understood. Variational equations play a central role in many aspects of dynamical systems, including stability theory, perturbation analysis and differential-Galois approaches to integrability. For integrable partial differential equations, they also provide a natural linearization of the dynamics around distinguished solutions and are closely related to the structure of conservation laws and squared eigenfunction expansions.

For the KdV equation, the importance of squared eigenfunctions was already recognized in the classical works of Gel'fand, Dikii, Gardner, Greene, Kruskal, Miura and many others. These functions constitute a fundamental tool in the inverse scattering formalism and in the construction of conserved quantities. On the other hand, the second symmetric power of the Schrödinger operator has long been known to occupy a privileged position in the theory of Lamé equations and finite-gap operators, beginning with the works of Hermite and Halphen. Nevertheless, the geometric mechanisms through which the second symmetric power, the variational equation and the spectral curves interact remain far from completely understood.

The main purpose of this paper is to exhibit the variational equation around a cnoidal solution of the KdV equation as a genuinely spectral-geometric object. Our starting point is the Lamé–Schrödinger operator naturally associated with the cnoidal wave. The corresponding spectral curve determines a spectral Picard–Vessiot structure whose eigenfunctions are given by the classical Hermite–Halphen solutions. We show that the second symmetric power of this operator provides a natural bridge between squared eigenfunctions, conserved densities of the KdV hierarchy and solutions of the adjoint variational equation supported on spectral curves.

A central observation of the present work is that the variational dynamics is governed not only by the classical spectral curve of the Schrödinger operator, but also by a second spectral object associated with the third-order operator arising in the variational equation itself. Using recent results on the factorization of algebro-geometric third-order operators, we construct the corresponding spatial spectral curve and obtain explicit solutions of the variational equation parametrized by points of this curve. Accordingly, the variational problem admits two complementary geometric descriptions: one inherited from the Schrödinger spectral curve through the second symmetric power construction, and another arising intrinsically from the centralizer of the third-order variational operator.

The interaction between these two spectral geometries constitutes one of the main themes of this work. In particular, we show that three apparently different constructions — Hermite–Halphen eigenfunctions, squared-eigenfunction expansions and factorization over the spectral curve of the third-order operator — lead to compatible families of solutions. This compatibility becomes especially transparent at the branch points of the Schrödinger spectral curve, where all constructions degenerate in a highly structured manner.

From a geometric viewpoint, these branch points play a distinguished role. They correspond to the Lamé solutions of the Schrödinger operator and generate static solutions of the variational equation. Moreover, the degeneration phenomena observed at these points suggest a deeper relationship between ramification loci on spectral curves and the differential-Galois structure of the associated linear problems. This observation motivates a conjectural picture in which branch points arise naturally as loci where the corresponding differential-Galois groups undergo a geometric degeneration.

The results obtained in this paper therefore connect several classical and modern themes: finite-gap integration, spectral curves, commuting differential operators, Picard–Vessiot theory, variational equations and conservation laws of the KdV hierarchy. Beyond the particular case of cnoidal waves, the methods developed here provide evidence that variational equations associated with finite-gap solutions should themselves admit a rich spectral-geometric interpretation.

\medskip

\subsection*{Main Results}

The main goal of this work is to show that the variational equation around a KdV cnoidal wave admits a natural spectral-geometric description governed by two interacting spectral curves. The first one is the classical spectral curve associated with the Lamé--Schrödinger operator defining the finite-gap background, while the second one arises from the spectral theory of the third-order operator governing the variational dynamics itself.

Our first result establishes that these two spectral structures are linked through the second symmetric power construction.

\medskip

\noindent{\bf Theorem A.} [Spectral-geometric description of the variational equation]

{\it 
Let \( 
u_0(z)=2\wp(z;g_2,g_3)-\frac{c}{6} \) be a cnoidal stationary solution of the KdV equation and let
\[
L=-\partial_z^2+u_0
\]
be the associated Lamé--Schrödinger operator with spectral curve \( 
\Gamma_0 \). Let \(  M=-\partial_z^3+(6u_0+c)\partial_z+6u_{0,z} \) be the third-order operator defining the variational equation
\[
\xi_t=-M\xi .
\]
Then:

\begin{enumerate}
\item The second symmetric power of $L-E$ coincides with the spatial part of the adjoint variational equation and transforms pairs of eigenfunctions of the Schrödinger problem into solutions of the adjoint variational equation.

\item The variational equation admits a spectral representation over the spectral curve ${\Gamma_1}$ associated with the centralizer of $M$.

\item The solutions obtained from the Lamé spectral curve $\Gamma_0$ through squared eigenfunctions and those obtained from the spectral curve ${\Gamma_1 }$ through factorization of $M-\nu$ define compatible families of variational solutions.
\end{enumerate}
}

Consequently, the variational dynamics of the cnoidal wave is governed simultaneously by the geometry of the spectral curves $\Gamma_0$ and ${\Gamma_1 }$.

\medskip

The second result identifies a distinguished geometric role for the ramification locus of the Lam\'e spectral curve $\Gamma_0$.

\medskip

\noindent{\bf Theorem B.} [Ramification locus and unification of spectral constructions]
{\it 

Let \( \Gamma_0 \) be the spectral curve of the Lamé--Schrödinger operator, and let
\( Z=\{B_1,B_2,B_3\} \) denote its set of branch points. Then:
\begin{enumerate}
\item At the points of $Z$, the Hermite--Halphen eigenfunctions degenerate into the classical Lam\'e solutions.

\item The squared-eigenfunction construction produces static solutions of the adjoint variational equation, whose derivatives generate distinguished static solutions of the variational equation.

\item The spectral factorization of the third-order variational operator leads to the same family of solutions.

\end{enumerate}
}

Therefore, the branch points of $\Gamma_0$ constitute the locus where the Hermite--Halphen construction, the squared-eigenfunction formalism, and the spectral factorization of the variational operator become unified. In particular, the ramification locus acquires a distinguished geometric role in the spectral description of variational dynamics.

\medskip

Taken together, these results show that the variational equation around a KdV cnoidal wave is not merely a linearization of the nonlinear dynamics, but rather a spectral-geometric object controlled by the interaction of two spectral curves whose constructions become unified at the ramification locus.

This interaction between the Lamé spectral curve and the spectral curve of the variational operator is the central geometric theme of the present work.

\bigskip

{\it The paper is organized as follows.} Section \ref{sec-preliminaries} reviews the spectral Picard–Vessiot description of the Lam\'e–Schr\"odinger operator associated with the cnoidal wave and recalls the structure of its spectral curve $\Gamma_0$. Section \ref{sec-preliminaries} contains the proof of Theorem B.1. Section \ref{sec-sym-power} studies the second symmetric power of the Schr\"odinger operator and establishes its relationship with squared eigenfunctions, conserved densities, and the KdV hierarchy. The asymptotic expansion used is slightly different from the classical one. Section \ref{sec-sym-power} contains the proof of Theorem A.1, Theorem B.2. Section \ref{sec-VE} is devoted to the variational equation around the cnoidal wave. We derive explicit solutions using both the Schr\"odinger spectral curve $\Gamma_0$ and the spectral curve $\Gamma_1$ associated with the corresponding third-order operator, and we analyze the special role played by the branch points. Section \ref{sec-VE} contains the proof of Theorem A.2,  A.3, and Theorem B.3. Several technical results concerning KdV differential operators, Lam\'e solutions, and third-order algebro-geometric operators are collected in the appendices.

\section{Preliminaries}\label{sec-preliminaries}

The KdV hierarchy provides a whole family of differential equations in partial derivatives. The first of these is called the Korteweg-de Vries equation. It was discovered in 1895 by Korteweg and de Vries, and its form is as follows:
{\begin{equation}\label{eq:KdV}
    u_t +6 u u_x +u_{xxx} =0
\end{equation}
}

Next, we present the general solution of \eqref{eq:KdV} in the travelling wave regime. We follow  the presentation given by  N. A. Kudryashov in \cite{kudryashov2009}.

In variables $(z,t)=(x-wt,t)$, we look for solutions $u(x,t)=y(z)$. Then integrating \eqref{eq:KdV} with respect to $z$, this equation  is the nonlinear ordinary differential equation
\begin{equation}\label{eq:d-TW}
    y_{zz}+3y^2-wy+C_1 =0
\end{equation}
with $C_1$ an integration constant. Multiplying equation \eqref{eq:d-TW} by $y_z$ and integrating with respect to $z$, we obtain:
\begin{equation}\label{eq:e-TW}
    (y_z )^2 +2y^3-wy^2+2C_1 y+2C_4 =0 ,
\end{equation}
where $C_4$ is a second integration constant. The polynomial $F(y)= y^3 -\frac{w}{2}y^2+C_1 y +C_4 $  allows us to write the differential equation in terms of its roots $\alpha$, $\beta$, $\gamma$. Specifically, $(y_z )^2 =-2 ( y-\alpha)(y-\beta)(y-\gamma) $. The real case can be consider, and in this setting is usually assumed $\alpha>\beta>\gamma$.

The general solution of the equation \eqref{eq:e-TW}  can be defined in terms of the Jacobi elliptic function $\cn (u,k)$, the elliptic cosine function, \cite{kudryashov2009, drazin1989},
\begin{equation*}
    y(z)= \beta +(\alpha-\beta)\cn^2 \left(
    u , k
    \right)\quad , \quad u= \sqrt{\frac{\alpha-\gamma}{2}}z \ , \ k^2 =\frac{\alpha-\beta}{\alpha-\gamma}
\end{equation*}
and $y(z)$ is called a {\it cnoidal wave} since it is given by the $\cn $ function. Pro\-perties of this special function, see \cite{chandrasekharan2012elliptic}, allows to rewrite $y(z)$in terms of the Weierstrass $\wp (z)$. In fact, defining 
\begin{equation*}
    \widetilde{y}(z):= y(- \imath\sqrt{2} z )-\frac{w}{6} ,
\end{equation*}
we obtain the differential equation
\begin{equation}\label{eq-WP}
    \left(
    \widetilde{y}'
    \right)^2 =
    4\widetilde{y}^3- \widetilde{g_2 } \widetilde{y} - \widetilde{ g_3 } \ ,
\end{equation}
where 
\begin{equation*}
  \widetilde{g_2 } = \frac{1}{3}w^2 -4C_1 \quad , \quad   
  \widetilde{ g_3 } = \frac{w^3 }{27} - \frac{2w}{3}C_1 -4C_4 \ .
\end{equation*}
Consequently, we obtain the solution of \eqref{eq-WP} 
\[
 \widetilde{y}(z) = \wp(z; \widetilde{g_2 } ,  \widetilde{ g_3 } ).
\]
Let $e_1 $, $e_2 $, $e_3 $ be the roots of the polynomial $G(\widetilde{y}) =4\widetilde{y}^3-\widetilde{g_2 } \widetilde{y} - \widetilde{ g_3 }$. Then
\begin{equation*}
    e_i = \int_{\wp(\omega_i )}^\infty  \frac{dt}{\sqrt{4t^3 -\widetilde{g_2 } t -\widetilde{ g_3 } }} \ , \ \textrm{for } i=1, 2, 3.
\end{equation*}
In this setting $\omega_1$, $\omega_2$ are the periods of $\wp (z)$, and $\omega_3 =\omega_1 + \omega_2$. Thus,  we obtain
\begin{equation*}
    y(z)=\widetilde{y}\left(\frac{\imath}{\sqrt{2}}z 
    \right)+\frac{w}{6} = 
    \wp\left(\frac{\imath}{\sqrt{2}}z; \widetilde{g_2 }  , \widetilde{g_3 }  \right) +\frac{w}{6} 
    .\end{equation*}
Properties of the $\wp$ function, see \cite{pastras2020weierstrass}, p 13, allow to express the cnoidal wave as

\begin{equation*}
    y(z)=  
    2\wp\left(\imath z; \frac{\widetilde{g_2 }}{4}  , \frac{\widetilde{g_3 }}{8}  \right) +\frac{w}{6} 
    .
\end{equation*}
To shorten we  write $g_2= \frac{\widetilde{g_2 }}{4}$ and $ g_3 =\frac{\widetilde{g_3 }}{8}$. Thus, we obtain the rewriting of the traveling wave solution of  \eqref{eq:KdV}:

\begin{equation*}
     u(x,t)=y(z)= -2\wp\left( 
    z; {g_2 }  , -{g_3 }  \right) +\frac{w}{6} .
\end{equation*}
Let define $\widetilde{z}=x+c t$. It follows from the preceding discussion that the function 
\begin{equation}\label{def-cnoidal-potential}
    u_0 (\widetilde{z})=u(\imath x, \imath t) =-2\wp\left( 
    \imath \widetilde{z}; {g_2 }  , -{g_3 }  \right) -\frac{c}{6}=2\wp\left( 
      \widetilde{z} ; {g_2 }  , {g_3 }  \right) -\frac{c}{6} .
\end{equation}
satisfies the following $\KdV_1$ ordinary differential equation:
{\begin{equation}\label{eq:KdV-1}
     u_{\widetilde{z}\widetilde{z}\widetilde{z}}-6 u u_{\widetilde{z}} -cu_{\widetilde{z}} =0 .
\end{equation}
}Using the notation in the appendix \ref{sec-KdV-polinomials}, we can rewrite it as  $\KdV_1 (u_0 , c/4)=0$. Then, the differential operator 
\begin{equation*}
 \hat{P}_{3} =P_3+\frac{c}{4}P_1 = -\partial^3+\frac{3}{2}  u_0 \partial +\frac{3}{4} u_0' + \frac{c}{4}\partial = -\partial^3 + \left(  \frac{3}{2}  u_0 + \frac{c}{4}\right) +\frac{3}{4} u_0'.   
\end{equation*}
commutes with the Schr\"odinger operator $L=-\partial^2 + u_0$, since $[ 4\hat{P}_{3} , L]=4 \cdot \KdV_1 (u_0 , c/4)=0$.

\bigskip

\centerline{\rule{5cm}{0.4pt}}

\bigskip

Now we consider the following KdV equation where $\KdV_1$ is defined in Apendix \ref{sec-KdV-polinomials} formula \eqref{eq-KdV},
{\begin{equation}\label{eq:KdV-1}
    u_t = 4\cdot \KdV_1 (u , c/4) = u_{zzz}-6 u u_z -c u_z.
\end{equation}}
which is satisfied by the cnoidal wave
\begin{equation}\label{eq-cn-wave}
u_0(z)=2\wp(z)-\frac{c}{6}    .
\end{equation}
In the remainder of this work, the analyzed KdV equation will be the one specified in \eqref{eq:KdV-1}. 

Consider the following ordinary differential operators
\begin{equation}\label{eq:Lax-c-cnoidal}
   L=  -\partial^2 +u  \ , \     \widetilde{A}=4\cdot   \hat{P}_{3}=-4\partial^3 +(6u +c)\partial + 3 u_{z}.
\end{equation}
Then equation \eqref{eq:KdV-1} appears as their compatibility condition. Accordingly, this relationship can be written as the following Lax equation:
{\begin{equation}\label{eq:Lax-operatorEQ-cnoidal}
   L_t = [ \widetilde{A}  ,L].
\end{equation}}
In the search for  solutions to the spectral problem associated with the operator Schr\"odinger $L$, the induced time evolution is determined by the operator 
$\widetilde{A}$. Therefore, we have to study the system

{\begin{equation}\label{eq:t-spectral}
(P1)\qquad    L\phi =E\phi \quad , \quad \phi_t =\widetilde{A} \phi   .
\end{equation}}

Moreover, the cnoidal wave $u_0(z)=2\wp(z)-\frac{c}{6}   $ with $\wp(z)$  the Weierstras function $\wp (z; g_2 , g_3 )$ is associated to the elliptic curve $\cE$ defined by the polynomial $Y^2 = 4X^3-g_2 X -g_3$. We denote by $\omega_1$, $\omega_2$ , $\omega_3 := \omega_1 +\omega_2$ its half-periods. Then
$$
e_1 = \wp (\omega_1 ; g_2 , g_3 ) \, ,\, 
e_2 = \wp (\omega_3 ; g_2 , g_3 ) \, ,\, 
e_3 = \wp (\omega_2 ; g_2 , g_3 ) \, ,\, 
$$
are the roots of the polynomial $4X^3-g_2 X -g_3$, and the half-periods are the roots of $\wp'$ in a fundamental period paralelogram $\Omega$,
\begin{equation}\label{def-ceros-wpp}
    \wp' (\omega_i ; g_2 , g_3 ) = 0  \  , \ i = 1,  2 , 3, 
\end{equation}
see for instance \cite{pastras2020weierstrass}. Moreover, the Weierstrass elliptic function obeys the homogeneity relation 
\begin{equation}\label{elliptic_homogeneity}
 \wp (z; g_2 , g_3 )= a^2 \wp (a \cdot z ; \frac{g_2 }{a^4 } , \frac{g_3 }{a^6 } )
\quad , \quad 
\wp' (z;g_2 , g_3 )= a^3 \wp' (a \cdot z ; \frac{g_2 }{a^4 } , \frac{g_3 }{a^6 } ).   
\end{equation} 
Then, in particular, for $a=\imath$ this gives $ \wp (z; g_2 , g_3 )= - \wp (\imath  \cdot z ; {g_2 } , -{g_3 } )$ and $ \wp' (z; g_2 , g_3 )= - \imath \cdot \wp (\imath  \cdot z ; {g_2 } , -{g_3 } )$.

The elliptic function $u_0$   in \eqref{eq-cn-wave} satisfies   the ordinary differential equation \eqref{eq:KdV-1} in the stationary case.  For the stationary solutions of $\KdV_1$,  the eigenvalues problem \eqref{eq:t-spectral} can be  algebraically solved. In fact, for a solution in  separate variables 
{\begin{equation}
\phi= e^{\mu t} \varphi (z) ,    
\end{equation}}
we obtain the following coupled eigenvalue problem  
{\begin{equation}\label{eq:s-KdV1}
(P2)\qquad     L\varphi =E\varphi \quad , \quad \widetilde{A} \varphi = \mu \varphi \quad \textrm{with } [L, \widetilde{A}]= 0 ,
\end{equation}}
since $u_0$ satisfies \eqref{eq:KdV-1}.
Moreover, it can be  algebraically solved, \cite{MRZ2}, to provide a solution over a curve $\Gamma$, the spectral curve associated to  \eqref{eq:s-KdV1}. In section \ref{sec-B} we summarize several relevant results concerning this outcome for ease of reference.

\subsection{Spectral Picard-Vessiot  approach to the Sch\"odinger operator}\label{sec-B}

The stationary cnoidal wave introduced in the previous section gives rise to a Lam\'e–Schr\"odinger operator whose spectral theory admits a natural algebro–geo\-metric description. Our aim in this section is to revisit this classical picture from the perspective of spectral Picard–Vessiot theory, following the framework developed in \cite{MRZ2}. This point of view will be essential in later sections, where the same geometric structures reappear in the study of the variational equation.

For the stationary solution $u_o (z)= 2\wp(z)-\frac{c}{6}$, the Lax pair \eqref{eq:Lax-c-cnoidal} can be rewritten in matrix  form, see  \cite{GH}, 
\begin{equation} \label{eq:cn1LAXpair}
\bm\psi_z  =  U_0 {\bm\psi} \quad  , \quad \mu\bm\psi   =  W_0  \bm\psi \,
\end{equation}
where \begin{equation} 
 \begin{array}{ll}
 U_0 & =  \begin{pmatrix}0&1\\ 2\wp(z)-\frac c6-E&0\end{pmatrix}   , \\
 W_0  & =  \begin{pmatrix}-2\wp'(z)&4\wp(z)+\frac {2c}3 +4E\\
4\wp(z)-4\wp^2(z)-4E^2-\frac 43cE-\frac{c^2}9-g_2&2\wp'(z)\end{pmatrix}  .
\end{array}  
\end{equation}
The compatibility condition between these equations is equivalent to the stationary KdV equation satisfied by $u_0$. Consequently, the pair$(U_0 ,W_0)$ defines an algebro–geometric spectral problem.

The corresponding spectral curve is obtained from the characteristic polynomial of $W_0$,
  \cite{GH},
\begin{equation}\label{eq-Gamma}
  f( E,\mu )= \det (W_0 -\mu)= \mu^2 -P(E)  
\end{equation}
whith \(P(E)=-16E^3-8cE^2+(12u_0 ^2+4cu_0-4u_{0,zz}-c^2)E+4u_0^3 + 4cu_0^2 +  c^2 u_0 - 2u_0 u_{0,zz} + u_{0,z}^2- cu_{0,zz} \).
This yields the affine algebraic curve
\begin{equation}
\Gamma \ : \  {\mu}^{2}+16\,{E}^{3}+8\,c{E}^{2}+ \left( \frac43\,{c}^{2}-4\,{ g_2 } \right) E+4\,
{ g_3 }+{\frac {2\,{c}^{3}}{27}}-\frac23 \,c{ g_2 }  ,
\end{equation}
The coefficients of this polynomial are first integrals of the stationary KdV flow and encode the spectral geometry of the cnoidal wave.

It is convenient to shift the spectral parameter by introducing
\[
\Tilde{E} :=E+\dfrac c6 .
\]
With this normalization the Schr\"odinger operator takes the standard Lam\'e form
\begin{equation}\label{ecu-Lame}
   L-E=   -\partial^2 +2\wp(z) - \widetilde{E}  
\end{equation}
and, abusing notation, the spectral curve becomes
\begin{equation}\label{eq:cnsp}
\widetilde{\Gamma}: \ \mu^2=-16\Tilde{E}^3+4g_2 \widetilde{E}-4g_3 . 
\end{equation}
This is precisely the elliptic curve naturally associated with the cnoidal wave.

We now reinterpret the classical Hermite–Halphen construction from the
viewpoint of spectral Picard–Vessiot theory. The spectral curve
$\widetilde{\Gamma}$ not only parametrizes the eigenvalues of the
Lam\'e–Schr\"odinger operator, but also provides a natural geometric setting
for describing its factorization and its families of eigenfunctions. In the language of spectral Picard–Vessiot theory, this family defines an eigenvector bundle over the curve.

The first step consists in identifying the ramification structure of $\widetilde{\Gamma}$. The branch points are
\[
B_i=(\widetilde{ e_i },0)=(-e_i,0),
\qquad i=1,2,3,
\]
where $e_1, e_2 , e_3$ are the roots of
\[
4X^3 - g_2 X - g_3.
\]
Let \(Z=\{B_1,B_2,B_3\}\) denote the ramification locus.
Away from \(Z\), the spectral curve is uniformized by the elliptic
parameter \(s\) through 
\begin{equation*}
  \Tilde{E}(s):= - \wp (s; g_2 , - g_3 ) \quad , \quad   \mu(s):= 2 \wp' (s; g_2 , -g_3 ).
\end{equation*}
which uniformizes $\widetilde{\Gamma}$ and allows the spectral parameter to be replaced by $s$ on the elliptic curve.

An important geometric feature of the spectral curve
\(\widetilde{\Gamma}\) is the hyperelliptic involution
\[
\iota:\widetilde{\Gamma}\longrightarrow\widetilde{\Gamma},
\qquad
(\widetilde E,\mu)\longmapsto(\widetilde E,-\mu),
\]
which exchanges the two sheets of the double covering
\(\widetilde{\Gamma}\to\mathbb C\),
\((\widetilde E,\mu)\mapsto \widetilde E\).
In terms of the uniformizing parameter \(s\), this involution is simply

\[
s\longmapsto -s,
\]
since the Weierstrass function is even whereas its derivative is odd:

\[
\widetilde E(-s)=\widetilde E(s),
\qquad
\mu(-s)=-\mu(s).
\]
The branch points \(B_i=(\widetilde e_i,0)\) are precisely the fixed
points of \(\iota\). Geometrically, they correspond to the points where
the two sheets of the spectral curve coalesce and where the spectral
parameter ceases to distinguish the two local eigenfunctions.

For every point
\(\cN (s)=(\widetilde E(s),\mu(s))\in\Gamma\setminus Z\),
the Schr\"odinger operator $L-E$ admits a factorization over the spectral curve, 
\[
  L-E =(-\partial -\Tilde{\sigma} )(\partial -\Tilde{\sigma}) 
\]
with
\begin{equation}\label{eq1-factor-L-cnoidal}
\Tilde{\sigma} (z,s) =\frac {\mu(s) -   2\wp'(z)  }{4\wp(z)
-\frac13 c+4\,  \Tilde{E}(s) } 
=\frac {-\imath \cdot  \wp'(\imath \cdot s )  -   \wp'(z)  }{2\wp(z)
 + 2\,  \wp(\imath \cdot s ) -\frac16 c} .
\  , 
\end{equation}
This factorization immediately produces two solutions of the Lamé
equation. Indeed, if
\[
(\partial-\widetilde\sigma_i)\varphi_i=0,
\qquad
\widetilde\sigma_i=
\widetilde\sigma\bigl(z,(-1)^is\bigr),
\tag{31}
\]
then \(\varphi_1\) and \(\varphi_2\) are linearly independent solutions of
\(L-E\). Their Wronskian satisfies
\[
\frac{\operatorname{wr}(\varphi_1,\varphi_2)}
{\varphi_1\varphi_2}
=
\widetilde\sigma_1-\widetilde\sigma_2
=
\frac{\mu(s)}
{2\wp(z)-\frac{c}{6}+2\widetilde E(s)}
=
\frac{-i\,\wp'(is)}
{\wp(z)-\frac{c}{12}+\wp(is)},
\]
which is nonzero for every point of
\(\widetilde{\Gamma}\setminus Z\).
Hence \(\varphi_1\) and \(\varphi_2\) form a basis of solutions of the Lamé
equation away from the ramification locus. They are the classical Hermite–Halphen functions.

The branch points are closely related to the half-periods of the
underlying elliptic curve. Since the zeros of \(\wp'(z)\) occur at the
half-periods \(\omega_i\), and
\(\widetilde e_i=-e_i\), the corresponding points of
\({\Gamma}\) can be identified with
\(\widetilde\omega_i=-i\omega_{\varepsilon(i)}\) for a suitable
permutation \(\varepsilon\) of \(\{1,2,3\}\).

Consequently, if \(\Omega\) denotes a fundamental period parallelogram,
a basis of solutions of
\[
L-E
=
-\partial_z^2+2\wp(z)-\wp(is)
\]
is given by the classical Hermite–Halphen functions
\[
\varphi_i(z,s)
=
\frac{\sigma\!\left(z+(-1)^{i+1}s\right)}
{\sigma(s)\sigma(z)}
\,e^{(-1)^iz\zeta(s)},
\qquad
P_i=\mathcal N((-1)^is)\in
\widetilde\Gamma\setminus Z,
\tag{32}
\]
where \(\sigma\) and \(\zeta\) denote the Weierstrass sigma and zeta
functions, respectively. These functions constitute the geometric
eigenfunctions naturally associated with the spectral curve and provide
the local sections of the spectral bundle that will play a central role
throughout the remainder of the paper.

The involution \(\iota\) has a direct interpretation at the level of
eigenfunctions. The two factorizations associated with the points
\(P=\mathcal N(s)\) and \(\iota(P)=\mathcal N(-s)\) give rise to the two
Hermite–Halphen solutions of the Lamé equation. Thus, the pair
\((\varphi_1,\varphi_2)\) may be viewed as the realization, at the level of
solutions, of the two sheets of the spectral covering. Away from the
ramification locus these solutions are distinct and linearly
independent, whereas at the branch points they merge into the classical
Lam\'e solutions.

The pair of Hermite–Halphen solutions provides more than a basis of the
scalar Lamé equation. Indeed, they assemble naturally into a fundamental
matrix of the first-order system \eqref{eq:cn1LAXpair}. For each point
\(P=\mathcal N(s)\in\widetilde\Gamma\setminus Z\), define

\[
\Psi(z,s)
=
\begin{pmatrix}
\varphi_1(z,s) & \varphi_2(z,s)\\
\varphi_1'(z,s) & \varphi_2'(z,s)
\end{pmatrix}.
\tag{34}
\]

Since the Wronskian of \(\varphi_1\) and \(\varphi_2\) does not vanish on
\(\widetilde\Gamma\setminus Z\), the matrix \(\Psi(z,s)\) is invertible
and therefore constitutes a fundamental matrix of the Lax system \eqref{eq:cn1LAXpair}.
In this way, the family of matrices \(\Psi(z,s)\) may be viewed as a
matrix realization of the spectral bundle over the curve
\(\widetilde\Gamma \setminus Z\).

The involution of the spectral curve,
\[
(\widetilde E,\mu)\longmapsto (\widetilde E,-\mu),
\]
interchanges the two Hermite–Halphen solutions and hence exchanges the
columns of \(\Psi\). From this perspective, the fundamental matrix
inherits directly the geometry of the spectral curve.

Returning to the original KdV Lax system \eqref{eq:t-spectral}, the temporal evolution is
obtained by attaching the exponential factors determined by the spectral
parameter \(\mu\). Consequently, the corresponding fundamental matrix is

\[
\Phi(x,t,s)
=
\Psi(x-ct,s)
\begin{pmatrix}
e^{\mu(s)t} & 0\\
0 & e^{-\mu(s)t}
\end{pmatrix}.
\tag{35}
\]
Equivalently, using the explicit Hermite–Halphen expressions, one
obtains the closed-form representation.

From the Picard–Vessiot viewpoint, the fundamental matrix
\(\Psi(z,s)\) plays a distinguished role. For each point
\(P=\mathcal N(s)\in\widetilde{\Gamma}\setminus Z\),
its entries belong to the Liouvillian extension generated by the
Hermite–Halphen eigenfunctions, and therefore \(\Psi\) determines the
corresponding Picard–Vessiot extension of the Lamé equation.
Rather than considering a single differential field extension for a
fixed spectral value, spectral Picard–Vessiot theory organizes these
extensions into a family parametrized by the points of the spectral
curve. In this sense, the curve \(\widetilde{\Gamma}\) acts as a moduli
space for the Picard–Vessiot extensions of the Schrödinger operator,
while the family of fundamental matrices \(\Psi(z,s)\) provides a
concrete realization of the associated eigenvector bundle. The passage
from individual solutions to the geometry of the spectral curve thus
transforms the classical spectral problem into a geometric object whose
fibres are Picard–Vessiot solution spaces.

Consequently, the collection of fundamental matrices
\(\{\Psi(z,s)\}_{P\in\widetilde{\Gamma}\setminus Z}\)
may be regarded as a matrix-valued section of the spectral bundle.
The spectral Picard–Vessiot extension is obtained by adjoining the
entries of these matrices to the differential field
\(K(\widetilde{\Gamma})\), thereby producing a differential field whose
field of constants is precisely the function field
\(C(\widetilde{\Gamma})\). From this perspective, the spectral curve is
not merely the Burchnall–Chaundy curve of the commuting operators; it is
the geometric parameter space on which the Picard–Vessiot theory of the
Schrödinger operator naturally unfolds.

\section{Second Symmetric Power and Integrability of Schrödinger Ope\-ra\-tors}\label{sec-sym-power}

The relevance of the second symmetric power of linear second order differential equations is well-known  long time ago.  This tool  was used  with success in the seventies and eighties of the nineteenth century by Hermite  and Halphen to studding the Lam\'e equation \cite{HER,HALP} (see also \cite{WW, brezhnev2008does,smirnov2002elliptic})

\begin{rem}
The pair Lax \eqref{eq:Lax-c-cnoidal} will allow us to obtain  some solutions of the variational equations \eqref{eq:avesta} associated to the KdV equation \eqref{eq:KdV-1}.  They are the so-called ``squared eigenfunctions"\ of the Lax pair solutions, that is, the solutions of the linear ordinary differential operator called the second symmetric power. i. e. the linear operator satisfied by {the product of solutions $\psi:=\phi_i\phi_j$} of $L-E$. We follow the works \cite{mclaughlin1986construction,flaschka1980multiphase,mckean1975spectrum}, based on the seminal paper \cite{gardner1974korteweg}.    
\end{rem}

\bigskip

First we consider the second order formal Schr\"odinger operator on a differential field $(\K, \partial_z )$ containing $u$ with $\coC$ as field of constants,
\begin{equation}\label{eq-Schrodinger-E}
    L-E =  -\partial_{zz}+u -E ,
\end{equation}
where $E$ is taken as an algebraic indeterminate, that is, $\partial_z (E)=0$. We extend the differential ring $(\coC(u)[E], \partial_x )$  with two indeterminates $Y_0$, $Y_1$, and, abusing notation,  the extended derivation  
\begin{equation*}
    \partial_z (Y_0) = Y_1 \quad , \quad \partial_z (Y_1 )= (u-E)Y_0 .
\end{equation*}
Then, the second symmetric power of $L-E $, denoted by $ (L-E)^{\odot 2}$, is defined to be a nonzero element of the differential operator ring with coefficients in the differential ring $(\coC(u)[E], \partial_z )$ of minimal order such that $(L-E)^{\odot 2} (Y_0^2 ) =0 $. An algorithm is given to calculate the second symmetric power of an operator of order $2$, see Theorem 1 in \cite{bronstein1997symmetric}, see also \cite{singer1993galois} . This result applied to $L-E$ gives us the following equality
\begin{equation}\label{eq-secondSym}
(L-E)^{\odot 2}    =-\partial_{zzz} +4 (u-E)\partial_z +2u_z .
\end{equation}
Let $\mathcal{F}^\dag$  be  a differential closure (\cite{Kolchin1973}, p. 102) of the differential field $(\coC(u)(E), \partial_z )$, and $\phi_1 , \phi_2 \in \mathcal{F}^\dag$ two independent solutions of \eqref{eq-Schrodinger-E}. Then, by construction, $\psi_i := \phi_i^2$ is a solution of \eqref{eq-secondSym}. Thus, $\psi_{zzz}=4 (u-E)\psi_z +2u_z \psi  $.

\bigskip

For $u$ a  solution of \eqref{eq:KdV-1},   the time evolution of a solution  $\phi$ of the spectral problem \eqref{eq:t-spectral} is given by $ \phi_t=\widetilde{A}\phi =-4\phi_{zzz}+(6u+c)\phi_z+3u_z \phi$. Then we obtain that $\psi := \phi^2$ is a solution of the following linear differential system of PDE
\begin{equation}\label{eq:slax_KdV}
\left\{
\begin{array}{rr}
-\psi_{zzz}+4 (u-E)\psi_z +2u_z \psi  &=0 \, ,\\   
\psi_t-(2u+4E+c)\psi_z + 2u_z\psi &=0.
\end{array}
\right.
\end{equation}
Now, eliminating $E$ in the system  \eqref{eq:slax_KdV}, we obtain 
{\begin{equation}\label{ec-eliminarE}
\psi_t + \psi_{zzz} - (6u+c  )\psi_z =0 .    
\end{equation}
}

For ease of reference, we refer to equation \eqref{eq:slax_KdV} as {\it the Lax equation for the second symmetric power} \eqref{eq-secondSym}. The equation \eqref{ec-eliminarE} is the so called {\it adjoint variational equation},  equation  \eqref{eq:avesta} in Section \ref{sec-VE}. 
Therefore, in the previous notation, we obtain the following conclusion.

\begin{cor}\label{cor-sol-second-symm}
  Any closed form solution  $\varphi (z, P )$ of the spectral problem \eqref{eq:s-KdV1}, provides a closed form solution $ \psi = \phi^2 = e^{2\mu t} \varphi (z, P)^2 $ of the adjoint variational equation \eqref{ec-eliminarE}, with $P=(E, \mu )$ in the spectral curve $\Gamma$ defined in \eqref{eq-Gamma} and associated to \eqref{eq:s-KdV1}.
\end{cor}

System \eqref{eq:slax_KdV} yields an important differential consequence. Combining the second equation with twice the first gives: 
{
\begin{equation}\label{eq-generating-densities}
  \psi_t + \partial (-2\psi_{zz}-12E \psi +(6u -c)\psi)=0 \ .  
\end{equation}
}
We  call the resulting equation  the  {\it master equation of conserved densities} of  KdV.

\subsection{Formal Squared Eigenfunction Expansions and Conserved Densities in the KdV hierarchy}

Consider $u=u(z,t)$ a function in $1+1$ variables. In the $\Delta = \{ \partial=\partial_z , \partial_t \}$-field $K=\coC (z,t, u )$ with constants field $\coC$ of zero characteristic, we can consider a partial differential equation (PDE)
\begin{equation}\label{ec-on sol}
   \Sigma(z,t,u)=0 
\end{equation}
A {\it formal conservation law} of the PDE \eqref{ec-on sol} is an identity that holds whenever $u$ is a solution of \eqref{ec-on sol}, taking the form
\begin{equation}\label{eq-fluxes}
    \partial_t T + \partial_z X = 0,
\end{equation}
where $T$ and $X$ belong to the differential field ${K}$. Following  \cite{olver1993applications}, we call   $T$  a {\it conserved formal  density
 } and   $X$  its corresponding  {\it associated formal flux}.

\begin{rem} 
Two densities $\rho$ and $\widetilde{\rho}$ that differ by a total $z$-derivative,
$\widetilde{\rho} = \rho + \dz(\sigma)$, define the same conserved quantity. Conservation
laws are therefore considered equivalent up to total derivatives.
\end{rem}

\medskip
  
Next, we adopt the notation from the ring of differential polynomials given in Kolchin \cite{Kolchin1973}, p. 70. We now restrict our attention to the KdV hierarchy, whose elements lie in the (partial) differential polynomial ring  $R=\coC\{u\} $, that is,
 \begin{equation}\label{eq-sigma-ell}
   \Sigma_\ell(u):= u_t -\kdv_\ell  \quad , \quad  \Sigma_\ell(u)=0\ .   
 \end{equation}
In particular,
\begin{equation}\label{ecuacion-kdv1-MRZ}
  \Sigma_1 (u):= u_t -\kdv_1 = u_t - (-\frac{1}{4} u'''+\frac{3}{2} u u' ).
\end{equation}
Let $R[E]$ be the differential polynomial ring  in the algebraic variable $E$. Its constants ring is $\coC [E]$. For each $\ell \in \bbN$, $\ell\not=0$, we consider the $\Delta$ differential ideal generated by the differential polynomial $\kdv_\ell$. Let denoted by $I_\ell = [ \kdv_\ell ]$ and define the differential ring
 \begin{equation*}
     R_\ell = R/ I_\ell .
 \end{equation*}
The quotient map $\theta_\ell : R\rightarrow  R_\ell $ induces a ring homomorphism $R[E]\rightarrow  R_\ell [E]$, and it 
can be localized in the multiplicatively closed set of all powers of $E$, that is, $S= \{ E^{n} : n\in \bbN \}$. Then, we obtain the following homomorphism $\Theta_\ell : R[E]_E :=  S^{-1} (R[E]) \longrightarrow  S^{-1} ( R_\ell [E])$ given by 
\begin{equation}
   \Theta_\ell  \left( \frac{\sum_{i=0}^n p_i(u)E^{i}}{E^m }\right)= \frac{\sum_{i=0}^n \theta_\ell (p_i(u))E^{i}}{E^m }  \ .
\end{equation}

Recall that the equation $(\star )$ of the Lax's equation \eqref{eq:slax_KdV} for the second symmetric power of the operator $L-E =-\partial^2+u-E$ reads
 \begin{equation}\label{eq-master-densities}
-\psi_{zzz}+4 (u-E)\psi_z +2u_z \psi  =0 \quad .
\end{equation}

Next we proceed to solve \eqref{eq-master-densities} in the ring $S^{-1} ( R_\ell [E])$. To see that, first take 
\begin{equation}\label{P-ell}
Q_\ell = \sum_{n=0}^\ell \psi_n (2E)^{-(n+1)}    
\end{equation}
with $ \psi_n \in R$, and impose to satisfy \eqref{eq-master-densities}. Then, for the fixed $\ell >0$, we have to solve the differential recursion
\begin{equation} \label{recu-1}
 (\gamma_\ell )\left\{ \begin{array}{rl}
0&=-2 \psi_{0,z} \\
 \psi_{n+1,z}  &= \frac12 \left(-\psi_{n,zzz}+4 u\psi_{n,z} +2u_z \psi_n \right) , \  \textrm{for } n= 0, \dots , \ell -1 \\
0&=-\psi_{\ell,zzz}+4 u\psi_{\ell,z} +2u_z \psi_\ell 
\end{array}  
\right.
\end{equation}

In particular, defining $\psi_0 :=1$, we obtain $$\psi_{n+1,z} = 2\cR^* \psi_n$$ 
where $\cR^*$ is the formal adjoint  of the recursion operator $\cR$ of the KdV hierarchy, \eqref{eq-recursion}. In addition, the families $\{ \kdv_n \}$ and $\{ v_n \}$ of differential polynomials defined in Appendix \ref{sec-KdV-polinomials} satisfy:

\begin{equation}\label{eq-kdv-cd}
\kdv_0:=u',\,\,\, \kdv_n:=\cR(\kdv_{n-1}),\mbox{ for }n\geq 1.
\end{equation}
\begin{equation}\label{eq-fn-cd}
v_0:=1,\,\,\, v_n:=\cR^*(v_{n-1}),\mbox{ for }n\geq 1.
\end{equation}
\begin{equation}\label{eq-vkdv-cd}
2\partial(v_{n+1})=\kdv_n.
\end{equation}
Consequently $\psi_0=v_0$, and equations \eqref{recu-1} read 
\[
 \psi_{n+1}  =\partial^{-1} (2\cR \partial ) \psi_n = 2 \cR^* \psi_n = 2^{n+1} v_{n+1} .
 \quad , \quad 
\]
In particular, {$\psi_{\ell} =2^{\ell} v_{\ell}  $}, and then the differential restriction in \eqref{recu-1} reads
\[
 -\psi_{\ell,zzz}+4 u\psi_{\ell,z} +2u_z \psi_\ell = 4\cR \partial \psi_\ell = 2^{\ell+1}\cR (2\partial  v_{\ell})=  2^{\ell+1}\cR (\kdv_{\ell-1} )= 2^{\ell+1}\kdv_{\ell}.
\]
Now we obtain the following. The element of the ring $S^{-1} ( R_\ell [E])$
\begin{equation}
   \Theta_\ell  (Q_\ell ) =  \sum_{n=0}^\ell \theta_\ell (\psi_n ) (2E)^{-(n+1)} = \sum_{n=0}^\ell 
   2^{n}\theta_\ell ( v_{n} )  (2E)^{-(n+1)} = 
   \sum_{n=0}^\ell 
   2\theta_\ell ( v_{n} )  E^{-(n+1)}
\end{equation}
solves the differential recursions $(\gamma_\ell )$. Accordingly, 
\begin{equation}
\begin{array}{rl}
 \psi_0&=v_0=1,\\ 
\psi_1&=2 v_1 =u,\\ 
 \psi_2&=2^2 v_2 =-\dfrac12 u_{zz}+\dfrac32 u^2, \\ 
\psi_3&=2^3 v_3 = \frac52 u^3 -\frac52 u_{zz}u -\dfrac{ 5}{4}(u_z )^2  +\dfrac{ 1}{4} u_{zzzz},\\
\vdots
\end{array}
\end{equation}

We now proceed to show that the coefficients of $Q_\ell$ are a conserved formal  densities. We use equation \eqref{eq-P2n+1} in Appendix \ref{sec-algebro-L2} for that purpose to provide a differential consequence of the symmetric second power of the $L-E$ operator.

\medskip

First, observe that the quotient map $\theta_\ell : R\rightarrow  R_\ell $ induces a derivation in $R_\ell$. In fact, the morphism $\partial_{t_\ell } := \theta_\ell \circ\partial_t : R \rightarrow R_\ell$ satisfies the Leibniz rule. Moreover
\begin{equation}
    \partial_{t_\ell } (u) = u_t +I_\ell \ , \ \partial_{t_\ell }(\kdv_\ell ) =0\in R_\ell ,
\end{equation}
Then $\partial_{t_\ell }$ defines a  in the ring $R_\ell$, which we continue to denote by abusing the notation $\partial_{t_\ell }$. Then, the KdV equation \eqref{eq-sigma-ell} reads
\begin{equation}
    \partial_{t_\ell } (\Sigma_\ell )=0 \Longleftrightarrow u_t =0 \ \textrm{in } R_\ell.
\end{equation}
In other contexts,  this is formulated as
\begin{equation}
u_{t_\ell } = \kdv_\ell \ ,    
\end{equation}
and generalized to any KdV polynomial of level $\ell$, see Appendix \ref{sec-KdV-polinomials}.

We first consider the setting corresponding to the equation $u_{t_1 } = \kdv_1$. In the notations of  Appendix \ref{sec-KdV-polinomials}, the associated Lax equation is
{\begin{equation}\label{eq:Lax-kdv-l+1}
   L_{t_1 } = [ P_3  , L].
\end{equation}}
with 
\begin{equation}
P_{3}=\sum_{j=0}^1 \left(v_{1-j}\partial-\frac{1}{2}\partial(v_{1-j})\right)L^j.
\end{equation}

We employ the Lax representation associated with the preceding equation \eqref{eq:Lax-kdv-l+1}, which we derive by imposing that $\phi^2$ satisfies the time evolution equation, with $\phi$ such that $(-\partial^2 +u)\phi =E\phi$ for a $u$ a solution of $u_{t_1 } = \kdv_1$. Consequently, if $\psi=\phi^2$, its time evolution is governed by the following expression:
\begin{equation}
\psi_t =2\phi \phi_t = 2\phi P_{3}(\phi)= 2\phi\sum_{j=0}^1 \left(v_{1-j}\partial-\frac{1}{2}\partial(v_{1-j})\right)E^j \phi     
\end{equation}
Thus, we obtain
\begin{equation}\label{eq-Q-l-t}
   \psi_t = \sum_{j=0}^1 \left( 
   \partial(v_{1-j}\psi)-2\partial (v_{1-j})\psi 
   \right)
   E^j = 
   \partial(v_{1}\psi)-u'\psi 
    + 
   \partial(\psi)
   E.
\end{equation}
Now  take 
\begin{equation}\label{def-Q-ell} 
Q_\ell = \sum_{n=0}^\ell \psi_n (2E)^{-(n+1)}   = \frac12 \sum_{n=0}^\ell v_{n}  E^{-(n+1)} ,
\end{equation}
and  we obtain the following differential recursions forcing $Q_\ell$ verifying \eqref{eq-Q-l-t}.
\begin{equation}\label{recursion-densities-1}
    v_{n,t}= 
   \partial (v_{n+1} )+\frac12 \partial (uv_{n})-u' v_n \quad , \quad \textrm{for } n=0, \dots ,\ell .
\end{equation}
Consider the differential operator 
\begin{equation}\label{def-L}
 \cL(f):= \partial (\cR^* (f))+\frac12 \partial (uf) -u' f = (\partial \circ \cR^* + \frac12 \partial \cdot u\cdot  -u'\cdot )(f)   ,
\end{equation}
then the above equations equals
\[
\partial_t (v_n ) = \cL (v_n) \quad , \quad \textrm{for } n=0, \dots ,\ell . 
\]

The Lemma \ref{lem-v} implies that  $u' v_n$ is a total derivative, then $u' v_n =\D_x(q_n )$. Finally, we obtain equations \eqref{eq-fluxes} for the densities $T_n = v_n $ and the fluxes $X_n =-(v_{n+1} +\frac12  uv_{n}-q_n)$, since
\begin{equation*}
   v_{n,t}= 
   \partial \left( v_{n+1} +\frac12  uv_{n}-q_n \right) \quad , \quad \textrm{for } n=0, \dots ,\ell .   
\end{equation*}
From now on, we define $J_n := -2^{n} (v_{n+1} +\frac12  uv_{n}-q_n )$ and then the equations above are read as follows:
\begin{equation}\label{eq-time-v}
    \partial_t \psi_n + \partial J_n =0 \quad , \quad \textrm{for } n=0, \dots ,\ell . 
\end{equation}
In fact, for $u$ solving $u_t -\kdv_1 =0$, we have $Q_1 $ is a solution of $(\gamma_1 )$   and also $\partial_t \psi_1 + \partial J_1 =u_t -\kdv_1 =0$ with $J_1 = -2^{1}\cdot  \frac18 (u''-3u^2 )$. Consequently, equation \eqref{recursion-densities-1} is satisfied for $\ell=1$ in $R_1$, and also
\[
 v_{1,t_1}= 
   \partial \left( v_{2} +\frac12  uv_{1}-q_1 \right) \quad , \quad \textrm{with $t_1 = t $  in } R_1 .
\]
Moreover, equation \eqref{recursion-densities-1} provides the following evolution restrictions for $t=t_1$
\begin{equation}\label{eq-time-t1-v}
    \partial_{t_{1}} \psi_n + \partial J_n =0 \quad , \quad \textrm{for } n=2, \dots ,\ell , 
\end{equation}
for $\psi_n =\psi_n [u]$  on solutions $u$ of $u_{t_{1}} =\kdv_1$ in $R_1$.
\medskip

Assume now that we are concerned with the equation $u_{t_m } = \kdv_m$, the $m$-th element of the KdV hierarchy for $m>1$. Then the associated Lax equation is
{\begin{equation}\label{eq:Lax-kdv-l+1}
   L_{t_m } = [ P_{2m+1}  , L].
\end{equation}}
with 
\begin{equation}
P_{2m+1}=\sum_{j=0}^m \left(v_{m-j}\partial-\frac{1}{2}\partial(v_{m-j})\right)L^j 
\end{equation}
defined in Appendix \ref{sec-KdV-polinomials}. Now the time evolution is given by
\[
\psi_t =2\phi \phi_t = 2\phi P_{2m+1}(\phi)= 2\phi\sum_{j=0}^m \left(v_{m-j}\partial-\frac{1}{2}\partial(v_{m-j})\right)E^j \phi   =
\]
\[
\sum_{j=0}^m \left( 
   \partial(v_{m-j}\psi)-2\partial (v_{m-j})\psi 
   \right)
   E^j
\]
Consider $Q_\ell$ defined in \eqref{P-ell} for any positive integer  $\ell$, then $Q_m$ is the solution of $(\gamma_m )$ in $R_m$. Thus, forcing $Q_\ell$ in \eqref{def-Q-ell} to be a solution of the previous equation,  we obtain the $t=t_m$ evolution equating powers in $E$,
\[
 v_{n,t_m }= \sum_{\substack{(p,k)\in [0,m]\times [o,\ell ] \\ p+k=n}}
 \left(
 \partial (v_{m-p}v_k )-2\partial (v_{m-p})v_k 
 \right)
  \quad , \quad \textrm{for } n=0, \dots ,\ell .   
\]
on sololutions $u$ of $u_{t_m } = \kdv_m$. Consequently, for $\ell=m$ we get
\[
v_{n,t_m }=\sum_{p=0}^m  \left(
 \partial (v_{m-p}v_{n-p} )-\kdv_{m-p-1} v_{n-p} 
 \right)
  \quad , \quad \textrm{for } n=0, \dots ,m .   
\]
Appling Lemma \ref{lem-v}  in Appendix \ref{sec-KdV-polinomials}, we conclude that is $v_{n,t_m }$ a total derivative.
\begin{rem}
    In \cite{gel1979integrable}, equation \eqref{eq-master-densities} appears as equation (1.1), with no explicit connection established to the algebraic construction of the second symmetric power, as far as we are aware. Furthermore, adopting the notation of \cite{gel1979integrable}, it can be shown\footnote{The subscript $2j+1$ in $R_{2j+1}$ follows the convention of Gel'fand--Dikii~\cite{gel1979integrable} and refers to the order of the associated Lax operator $P_{2j+1}$, rather than indexing an independent sequence: $R_{2j+1}$ is thus the $j$-th coefficient in this generating series, labelled by the order of the differential operator it is naturally attached to, in parallel with our own use of $P_{2n+1}$ throughout the paper. In particular, it coincides (up to the sign and normalisation made explicit below) with the coefficient customarily denoted $R_j$ in the resolvent expansion $R_\lambda(z)\sim\sum_{\ell\ge 0}R_\ell[u]\,\lambda^{-\ell-1/2}$ of the diagonal Green's function of $L-\lambda$.} that
    \begin{equation*}
        v_j = (-1)^j 2 \overline{R_{2j+1}} \ .
    \end{equation*}
  where $\overline{R} = \sum_{j=1}^\infty \overline{R_{2j+1}} (-\imath )^{j}E^{-j/2}$ is a formal solution of the differential equation
   \begin{equation*}
       -2 y y''+(y'')^2+4(u-E)y^2=1 ,
   \end{equation*}
   obtained by formally integrating $\psi$
 times the equation \eqref{eq-master-densities} and setting the integration constant equal to 1.
\end{rem}

\subsubsection{On polynomial solutions of the second symmetric power}

In contrast to the constructions of Brezhnev~\cite{brezhnev2008does} and Smirnov~\cite{Smi94}, where finite-gap potentials and elliptic solutions of the KdV equation are determined explicitly by means of recursion relations on the $\Psi$-function and by ad hoc constructions, respectively, we propose here an alternative method grounded in differential algebra. Our approach recovers cnoidal (finite-gap) solutions without recourse to either the theta-function representation or the recursion formulae of~\cite{brezhnev2008does}, yielding instead a direct characterisation in terms of the conserved densities $v_j$. It thus provides a route complementary to that of Smirnov~\cite{Smi94}, one grounded instead in  differential algebra, and affords additional structural insight into the integrability of these potentials.

We now turn to establishing the conditions that guarantee the existence of a formal polynomial solution in $E$ to equation \eqref{eq-master-densities}; in fact, we shall obtain a complete characterisation of all such solutions in terms of the densities $v_j$ just introduced. To this end, we introduce the following notation.

For a polynomial $f(E)=\sum_{k=0}^\ell d_k E^k \in \coC[E]$ we define the sequence
\begin{equation}
    f_{[0]}(E) := f(E) \quad , \quad f_{[n]} := \frac{f_{[n-1]} (E)-f_{[n-1]} (0)}{E} \ , \ n=1, 2, \dots .
\end{equation}
Observe that $f_{[\ell]} (E)=d_\ell $\ and $ f_{[n]} (E) =0$ for $n>\ell$. From these observations, we obtain the following result (cf. \cite{gel1979integrable}, Prop. 1.1, where an analogous statement is attributed to Its and Matveev).

\begin{prop}
    Let $f$ be  a constant coefficients polynomial $f(E)=\sum_{k=0}^\ell d_k E^k $ in $\coC[E]$. Then
    \begin{equation}\label{def-Q-ell}
        Q_{\ell , f } (E) = f_{[0]}(E) v_0 +\cdots + f_{[\ell]}(E) v_\ell
    \end{equation}
    is a solution of \eqref{eq-master-densities} in $\coC\{u\}[E]$. Furthermore, any polynomial solution of \eqref{eq-master-densities} is of the form $Q_{\ell , f } $ for some $f  $ in $\coC[E]$.
\end{prop}

\begin{proof}
    Let $Q_\ell =\sum_{k=0} ^\ell \omega_k E^k$ be a polynomial solution of \eqref{eq-master-densities} in $\coC\{u\}[E]$ of degree $\ell >0$. Then equating degrees in $E$ we obtain the differential recursion
    \begin{equation}\label{recursion-D2}
        -\cD_2 (\omega_0 )=0 \ , \ -\cD_2 (\omega_n )=4 \partial (\omega_{n-1}) \ , \ 4 \partial (\omega_\ell )=0 \ .
    \end{equation}
    with $ \cD_2 = \dxxx - 4u\,\dx - 2u_x $. 

For $ Q_{\ell , f } (E) = f_{[0]}(E) v_0 +\cdots + f_{[\ell]}(E) v_\ell$ to be a solution we rewrite it as
\[
 Q_{\ell , f } (E) = \sum_{n=0}^{\ell} \left( \sum_{k=0}^{\ell -n} d_{n+k}  v_k \right) E^n  = \sum_{n=0}^{\ell}  \omega_n  E^n 
\]
with $\omega_n = \sum_{k=0}^{\ell -n} d_{n+k}  v_k$. Then the sequence $\omega_0 , \dots , \omega_\ell$ satisfies the differential recursion \eqref{recursion-D2}, which implies that $ Q_{\ell , f } $ is a solution of \eqref{eq-master-densities}.

Next we proceed by induction on $n$ to prove that we can construct a finite sequence $c_\ell , \dots c_n$ of constants such that $Q_\ell =\sum_{k=0} ^\ell \omega_k E^k=   Q_{\ell , f }$ for $f(E)=\sum_{n=0}^\ell c_n E^n$. In fact, since $\partial (\omega_\ell )=0$, we have $\omega_\ell =c_\ell$ for some constant $c_\ell \in \coC$. Consequently $4 \partial (\omega_{\ell -1}) = -\cD_2 (c _\ell  )=2u_z c_\ell$. Thus, $\omega_{\ell -1}= c_\ell v_1 +c_{\ell -1} v_0 $ for some constant $c_{\ell -1} \in \coC $. Suppose now that the formula
    \[
    \omega_{n+1} = \sum_{j=0}^{\ell-n-1}c_{n+1+j} v_j \ 
    \]
    holds, and let us prove it for $n$ by induction. We use recursion \eqref{recursion-D2} for this purpose.
    \[
    4 \partial (\omega_{n})= -\cD_2 (\omega_{n+1} )=4\partial \cR^* (\omega_{n+1})=4\partial 
    \left(
    \sum_{j=0}^{\ell-n-1} c_{n+1+j} v_{j+1}
    \right) .
    \]
 Therefore, $\omega_{n} = \sum_{k=1}^{\ell-n} c_{n+k} v_k + c_n v_0 $ for some constant $c_n \in \coC$. Consequently, we can write $Q_\ell $ as $Q_{\ell , f}$ with $f(E) = \sum_{k=0}^\ell c_k E^k$ because
   \[
   Q_{\ell } (E) = \sum_{n=0}^{\ell} \sum_{k=0}^{\ell -n} c_{n+k}  E^n  v_k= \sum_{k=0}^{\ell} 
   \left(
   \sum_{j=0}^{\ell -k} c_{k+j}E^j 
   \right)
   v_k = \sum_{k=0}^{\ell} f_{[k]} (E) 
     v_k \ ,
   \]
  and we achieve the statement.
\end{proof}

\noindent{\bf Conclusion. } Next, we compile the solutions of \eqref{eq-master-densities} obtained from the Lax pair \eqref{eq:slax_KdV}. In fact, for any $\ell >0$ and any polynomial $f\in \coC[E]$ of degree $m$, we denote by 
\begin{equation}
    \psi_{\ell, f}(E) := \sum_{n=0}^\ell 2v_n \frac{1}{E^{(n+1)}}+f_{[0]}(E) v_0 +\cdots + f_{[m]}(E) v_m .
\end{equation}
Then we have obtained the following solutions
\begin{enumerate}[label=(\alph*)]
    \item $ \psi_{\ell, f}(E)$ is a solution of \eqref{eq-master-densities} in $S^{-1} R_\ell [E]$.
    \item The time evolution of $\psi_{\ell, f} $ is provided by \eqref{eq-time-v}. Then, $ \psi_{\ell, f}(E)$ is a solution of \(\psi_t-(2u+4E+c)\psi_z + 2u_z\psi =0\).
\end{enumerate}
Therefore, $\psi_{\ell, f}(E)$ is a solution of the Lax pair \eqref{eq:slax_KdV} in $S^{-1} R_\ell [E]$ .

\subsubsection{Hierarchy of commuting flows}
Next we revisited the Hamiltonian approach to
the KdV hierarchy in connection with spatially rapidly decaying solutions. Consider 
the  space $\cS $, the Schwartz space of rapidly decreasing real-valued
functions $ u : \bbR\rightarrow\bbR$, we will exhibit a symplectic structure  on $\cS\times\cS$ and
Hamiltonian functions $ H_n : \cS \rightarrow \bbR$ such that the nth equation \KdV \ takes on the
form $u_t =2 \partial (\frac{\delta H_n   }{\delta u })$.

We can define
\begin{equation}
    H_n [u]:=\int v_{n}[u] =\frac{1}{2^{n}}\int \psi_{n}[u]\quad , \quad \textrm{ for } n\in \bbN .
\end{equation}
Then, by \eqref{eq-time-v},
\[
\frac{dH_n }{dt} =\frac{1}{2^{n} }\int -\partial J_{n} =0 .
\]
From (\cite{GD, olver1993applications}) the above, all of the gradients 
$$
\dfrac{\delta H_{n+1}}{\delta u}= \; 
v_{n+1} [u]\quad , \quad  n=1,2,3,...
$$
are solutions of the adjoint variational equations \eqref{ec-eliminarE}, as happens in finite dimension dynamical systems: the gradients of the first integrals are solution  of the adjoint variational equation, ie equivalently, the linear part of any  first integral of the dynamical system around the particular solution $z_0=z_0(t)$ is a first integral of the variational equation (see \cite{AMV}, p. 221). We remark  that for Hamiltonian systems this is more or less related  to Ziglin's lemma,  \cite{zi1,MR} (see also  \cite{morales}).

\color{black}

Each $H_n$ defines a flow compatible with KdV. The  symplectic form is the antisymmetrical operator $J=2\partial_x$.  Then,
        \[
          u_{t_n} = 2\partial\,\fdiff{H_{n+1}}{u} = \cD_2\,\fdiff{H_{n}}{u}.
        \]
with $\cD_2$ defining the Magrid second Poisson structure \cite{Magri1978}. Observe that all of these flows commute pairwise, forming the \emph{KdV hierarchy}. The differential equations of the standard KdV hierarchy are then
\begin{equation} \label{eq:hiercs} 
u_{t_n}= 2\partial(v_{n+1}) = \frac{1}{2^n }\partial \psi_{n+1}, \, n=1, 2, 3, ...,
\end{equation}
being the KdV equation \eqref{ecuacion-kdv1-MRZ} the first member of the hierarchy
$$ 
u_{t_1}=\frac12 \partial\left(-\dfrac 12 u_{xx}+\dfrac 32 u^2 \right).$$
 The Hamiltonian field defined by $H_n$ is then
\begin{equation}
X_n(u)=\frac{1}{2^n }\partial \psi_{n+1} ,
\end{equation}
and its action on functionals $F[u]$ is  the functional

\begin{equation}
(X_nF)=\int \dfrac {\delta F}{\delta u} X_n(u)\, dx=  \frac{1}{2^{n+1 }} \int \dfrac {\delta F}{\delta u}2\partial\psi_{n+1}\, =\{ F, H_{n+1}\},
\end{equation}
being  Poisson bracket between two functionals,
\begin{equation} \{F,G\} =2\int \dfrac {\delta F}{\delta u}\partial\dfrac {\delta G}{\delta u}\, .
\end{equation}
Moreover, taking into account the Lemma \ref{lem-v}, we obtain
\begin{equation*} \{H_n , H_m \}=2\int \dfrac {\delta H_n}{\delta u}\partial\dfrac {\delta H_m }{\delta u} = \int v_n \; \kdv_{m-1} =0 .
\end{equation*}
Then all the members of the KdV hierarchy, are in involution. That means that  not only any of the functionals $H_n$ are first integrals of KdV equation, but are also first integrals of any other equation of the KdV hierarchy \eqref{eq:hiercs}.

The generalized KdV hierarchy is obtained by substituting in  \eqref{eq:hiercs} $\psi_n$ by the  generalized coefficients in \eqref{eq-master-densities}, ie,

\begin{equation} \label{eq:hierce} u_{t_n}=2\frac \partial {\partial x} \left(\psi_n+c_1\psi_{n-1}+\cdots c_{n-1}\psi_1+c_n\right), \, n=1, 2, 3, ...,
\end{equation}
with $c_1 , \dots ,c_n$ arbitrary constants.


  \begin{rem} For this remark we follow \cite{mckean1975spectrum,MCL}.
Over the years, a variety of methods have been employed to derive the conserved densities of the KdV equation.

In \cite{GGKM}, the recursion \eqref{recu-1} was obtained from the squared symmetrical Lax pair \eqref{eq:slax_KdV} by means of a Miura transformation combined with an asymptotic expansion in an auxiliary parameter $\epsilon$.

In \cite{mckean1975spectrum}, the recursion \eqref{recu-1} was derived through a less direct approach: rather than working from the squared symmetrical Lax pair \eqref{eq:slax_KdV}, the authors proceeded from a partition function associated with the diffusion-like equation $\dfrac{\partial \phi}{\partial \tau} + L\phi = 0$.
\end{rem}

\begin{rem}
    \textbf{Connection with the Inverse Scattering Transform (IST).}
        The squared eigenfunctions of the associated Schr\"{o}dinger operator (Lax
        operator) are precisely the gradients $\dfrac{\delta H_n}{\delta u}$,
        revealing the deep connection between the bi-Hamiltonian structure and the IST.
\end{rem}

\begin{rem}

  {\it From the above, all of the gradients
{
$$\dfrac{\delta H_n}{\delta u}=\psi_n, n=1,2,3,...$$
}
are solutions of the adjoint variational equations}, as happens in finite-dimensional dynamical systems: the gradients of the first integrals are solutions of the adjoint variational equation, ie equivalently, the linear part of any  first integral of the dynamical system around the particular solution $z_0=z_0(t)$ is a first integral of the variational equation (see \cite{AMV}, p. 221). We remark  that for Hamiltonian systems this is more or less related  to Ziglin's lemma,  \cite{zi1,MR} (see also  \cite{morales}). Now in the Hamiltonian case, using $J=\partial_x$ from the solutions of the adjoint variational equation, we can obtain solutions of the variational equation (see section \ref{sec-VE}), ie, {\it the Hamiltonian fields of the hierarchy

$$X_n=\dfrac\partial{\partial x}\left(\dfrac{\delta H_n}{\delta u}\right), \, n=1,2,3,...$$
 are solutions of the variational equation \eqref{eq:vesta} }, as also happens in finite dimensional Hamiltonian systems: the Hamiltonian fields of the first integrals (restricted to the particular solution $z=z_0(t)$) are solutions of the variational equation along $z_0$; this can be seen by interpreting the variational equation as Lie symmetries fields , ie, $$[X, \xi]=0,$$
 restricted to  $z_0$ (see \cite{BCRS}). Now the key point is the connection between some solutions of the variational equations and the Lax pair: {\it solutions of  the squared Lax pair \eqref{eq:slax_KdV}  (given by product of solutions $\psi=\phi_1\phi_2$ of the Lax pair \eqref{eq:slax_KdV})  are also solutions of the adjoint variational equation \eqref{ec-eliminarE} and $\psi_x$ becomes solutions of the variational equation \eqref{eq:vesta}}. But, in order to obtain, from the above, {\it all the solutions of the variational equations on a  space defined by suitable boundary conditions,  it is an open problem to  prove a  completeness statement}.

\end{rem}

The subsequent section is devoted to a thorough analysis of separable-variable solutions arising from the variational equation associated with the KdV equation \eqref{eq:KdV-1}.

\section{Variational Equation around Cnoidal Waves}\label{sec-VE}

This section is devoted to the construction of closed-form solutions to the variational equation of KdV equation \eqref{eq:KdV-1} around the cnoidal wave \eqref{def-cnoidal-potential}. Other authors have studied the variational equation by relating it to square functions. See, for example, \cite{AiraultMcKeanMoser1977, Kapitula2008, Sachs1983, HaragusSattinger1998, PegoWeinstein1994}. Our approach relies on seeking separable-variable solutions, which naturally leads to the analysis of a coupled spectral problem for a certain third-order operator — one that proves to be of algebro-geometric type in the case of the cnoidal wave. Crucially, the geometry of the spectral curve associated with this new spectral problem plays a determining role in the structure of the resulting solution formulas. The present study draws on previous collaborative work of the third author with S. Rueda \cite{ZurroSevilla, RZ2021factoring}.

\bigskip

We consider the KdV equation 
\(
u_t = u_{zzz}-6u u_z  -cu_z ,
\)
defined in \eqref{eq:KdV-1}. Let F(u) be the differential polynomial 
\[F(u)=u_{zzz}-6u u_z  -cu_z \]
Assuming $F: X \rightarrow \coC$ with $X$ a complex Banach space, we can compute the Gateaux derivative of $F$ at $u$ in a generic direction $\xi \in X$,

\[
dF(u ;\xi) = \lim_{\epsilon\rightarrow 0}\frac{F(u+\epsilon\xi)-F(u)}{\epsilon}
\]
Observe that $F$  is Fr\'echet differentiable, then it is also Gateaux differentiable, and its Fr\'echet and Gateaux derivatives agree. In particular, the Gateaux derivative of $F$ at $u_0 =2\wp(z)-\frac{c}{6}  $ is 
\[
dF(u_0 ;\xi) = \lim_{\epsilon\rightarrow 0}
\frac{F(u_0+\epsilon\xi)-F(u_0)}{\epsilon}=
\xi_{zzz}-(6 u_0 +c) \xi_z -6u_{0,z} \xi \ .
\]
Then, considering the Gateaux derivatives at $u_0$ in the KdV equation, we obtain the linear differential equation called {\it the variational equation (VE) around the solution $u_0$ }
\[
\xi_t = \xi_{zzz}-(6 u_0 +c) \xi_z -6u_{0,z} \xi \ .
\]
To examine the solvability of this equation, we will first reformulate it as follows. Let $L_0 = -\partial_z^2 +(6 u_0 +c)= -\partial_z^2 + 12 \wp (z)$, \ and $M=\partial_z L_0 $ be the third order operator 
\begin{equation}\label{eq-operator-M}
M={-\partial_z^3+(6u_0+c)\partial_z  +6{u_0}_z
  } = \partial_z L_0 \,  .
\end{equation}
Then, the variational equation can be rewritten as 
{ \begin{equation}\label{eq:vesta}
 \xi_t= - M\xi \ .
\end{equation}}
We define the {\it formal adjoint variational equation} 
\begin{equation}\label{eq:avesta}
 \psi_t = - M^\dag \psi  ,  
\end{equation}
with $M^\dag = - L_0 \partial_z  =\partial_z^3-(6u_0 +c)\partial_z$.
Therefore, we confront the study of third-order operators.
\begin{equation}\label{eq-operator-M-cnoidal}
M=-\partial_z^3+12\wp(z)\partial_z  + 12\wp'(z) 
\end{equation}
and
\begin{equation}\label{eq:avesta-cnoidal}
    M^\dag  = \partial_z^3  -12\wp(z)\partial_z .
\end{equation}
Therefore, equation  \eqref{eq:avesta} is \eqref{ec-eliminarE} for $u=u_0 = 2\wp(z)-\frac{c}{6}   $. Observe that to solve equations \eqref{eq:vesta} and \eqref{eq:avesta}, we can take into account the following relationship:
\begin{equation}
\begin{array}{ccc}
     \eqref{eq:avesta} & \overset{\partial_z} \longrightarrow & \eqref{eq:vesta}\\
    \eqref{eq:avesta} & \overset{\partial_z^{-1}}\longleftarrow &\eqref{eq:vesta}
\end{array}\label{equiv-VEs}
\end{equation}

\subsection{Lifting solutions of the second symmetric power}\label{sec-lifting}

The adjoint variational equation  $ \psi_t = - M^\dag  \psi  =-(\partial_z^3  -12\wp(z)\partial_z ) \psi$, \eqref{eq:avesta}, has the following system of squared eigenfunctions solutions
\begin{equation}
    \psi_1= e^{2\mu t}\varphi_1^2 (z)   \quad ,  \quad           \psi_2= e^{2\mu t}\varphi_2^2 (z)   \quad ,  \quad    
     \psi_3= e^{2\mu t}\varphi_1 (z)  \varphi_2 (z) \quad ,  \ \mu\not=0,
\end{equation}
where $\{ \varphi_1 , \varphi_2  \}$ is a fundamental system of solutions of the operator $L-E$ defined in \eqref{eq:s-KdV1}, and $\{ \phi_1 = e^{\mu t}\varphi_1 , \ \phi_2 = e^{\mu t}\varphi_2 \}$ solves \eqref{eq:t-spectral} in separate variables. Let $\Gamma$ be the spectral curve associated to the spectral problem \eqref{eq:s-KdV1}. Let $Z=\Gamma \cap \{ \mu=0 \}$ be the branching points of $\Gamma$.

\begin{prop}
The functions $\psi_1 , \psi_2 , \psi_3$ are  linearly independent over $\coC$ at each point of $\Gamma \setminus Z$.
\end{prop}

\begin{proof}
 Let us consider the Wronskian  matrix 
    \[
   W (\psi_1 , \psi_2 , \psi_3 ) = \begin{pmatrix}
      \psi_1 & \psi_2 & \psi_3  \\
    \partial  \psi_1 & \partial\psi_2 & \partial\psi_3  \\
     \partial^2 \psi_1 &\partial^2   \psi_2 & \partial^2  \psi_3  
    \end{pmatrix} = e^{2\mu t} \begin{pmatrix}
      \varphi^2_1 & \varphi^2_2 & \varphi_1   \varphi_2\\
    \partial  \varphi^2_1 & \partial\varphi^2_2 & \partial(\varphi_1   \varphi_2) \\
     \partial^2 \varphi^2_1 &\partial^2   \varphi^2_2 & \partial^2  (\varphi_1   \varphi_2)  
    \end{pmatrix} = e^{2\mu t} B
    \]
with $B$ the fundamental matrix of the second symmetric power of $L-E = -\partial^2 +u_0 -E$.  Then, taking the determinants, we obtain
\[
wr (\psi_1 , \psi_2 , \psi_3 ) =e^{6\mu t} \det (B) \not= 0 .
\]
\end{proof}

\begin{cor} The  variational equation $\xi_t= -M \xi$, \eqref{eq:vesta}, has a  three-dimensional $\coC$-linear space of solutions generated by 
\[
\widetilde{\psi_1 } = e^{2\mu t}\partial\varphi^2_1
\quad , \quad 
\widetilde{\psi_2 } =  e^{2\mu t}\partial\varphi^2_2
\quad , \quad 
\widetilde{\psi_3 } =  e^{2\mu t}\partial (\varphi_1 \varphi_2 )
\quad , 
\]
over the spectral curve $\Gamma \setminus Z$ associated with the spectral problem     \eqref{eq:s-KdV1}.
\end{cor}

\begin{proof}  Let us consider the Wronskian  matrix $ W (\widetilde{\psi_1 }  , \widetilde{\psi_2 }  , \widetilde{\psi_3 } ) $. Then, taking the determinants, we obtain
\[
\det 
\left(
 W (\widetilde{\psi_1 }  , \widetilde{\psi_2 }  , \widetilde{\psi_3 } ) 
\right) = 2\;e^{6\mu t } \: u_{0,z} \cdot \det B 
\]
with $B$ the fundamental matrix of the second symmetric power of $L-E = -\partial^2 +u_0 -E$.  
    
\end{proof}

\medskip

\noindent {\bf On the static solutions of the variational equation.} The adjoint variation equation has also solutions that are polynomials in $E$. In particular, for any $\ell >0$ and $f\in \coC[E]$ we defined in \eqref{def-Q-ell} the following solution of  \eqref{eq:avesta} which is a polynomial in $\coC\{u\}[E]$:
\[
  Q_{\ell , f } (E) = f_{[0]}(E) v_0 +\cdots + f_{[\ell]}(E) v_\ell
\]
Observe that $\partial_t (  Q_{\ell , f } (E)  )=0$ since $\partial_t (u_0 )=0$. Then
\[
\widetilde{Q_{\ell , f } }(E)=  \partial (  Q_{\ell , f } (E)  ) = f_{[1]}(E)  \partial v_1 +\cdots + f_{[\ell]}(E)  \partial v_\ell
\]
is a  static solution of the variational equation \eqref{eq:vesta}.

Another type of static solutions of \eqref{eq:vesta} appears when consider the Hermite solutions  \eqref{sol-L} of the Lam\'e equation \eqref{eq:hLame}. In fact, this correspond to the points in $Z$, the branching  points of $\Gamma$. In this case, the second symmetric power of the Lam\'e equation has solutions
\begin{equation}
    \psi_1= \varphi_1^2 (z)  = \wp(z, g_2 , g_3 ) +\Tilde{e_1 } \quad ,  \quad           \psi_2= \varphi_2^2 (z)   =\wp(z, g_2 , g_3 ) +\Tilde{e_2 } \quad ,  
\end{equation}
and 
\[
\quad    
     \psi_3= \varphi_1 (z)  \varphi_2 (z) =\sqrt{(\wp(z, g_2 , g_3 ) +\Tilde{e_1 } )( \wp(z, g_2 , g_3 ) +\Tilde{e_2 }) }
     \quad .  \ 
\]
Their Wronskian is
\[
wr (\psi_1 , \psi_2 , \psi_3 ) = -\frac14 \frac{(\wp'(z) )^3 ( \tilde{e_1}-\tilde{e_2 } )^3}{(\psi_1 \psi_2 )^{3/2}}\not=0 \ .
\]
Consequently, the static associate solutions to \eqref{eq:vesta} are no longer independent since $ \partial \psi_1 = \partial \psi_2$. Furthermore, denoting $\widetilde{\psi_i } = \partial \psi_i$  the linear independence of \(\tilde{\psi _{2}}\) and \(\tilde{\psi _{3}}\) is established by the \(2 \times 2\) minor within the Wronskian matrix, specifically across the second and third columns, which produces a determinant of
\[
-\frac14 \frac{(\wp'(z) )^3 ( \tilde{e_1}-\tilde{e_2 } )^2}{(\psi_1 \psi_2 )^{3/2}}\not=0 \ .
\]
\bigskip

In what follows, we present two classes of solutions to the variational differential equation using the recent spectral Picard-Vessiot techniques for the tird order operator, \cite{RZ2024}. The first type is obtained through a separation of variables approach, see Section \ref{sec-nu}, whereas the second one comprises what we refer to as {\it static solutions}, see Section \ref{sec-nu-0}. This new type of solution has the advantage of being generally a three-parameter solution, which expands the possibility of perturbing the conical wave. We provide some preliminary results for achieving these goals in Appendix \ref{sec-Order-3}.



\subsection{Solutions over the Spectral Curve}\label{sec-nu}

Let us consider the problem of solving the variational equation \eqref{eq:vesta} around the cnoidal wave $u_0 (z)=2\wp\left( 
      {z} ; {g_2 }  , {g_3 }  \right) -\frac{c}{6}$, 
 \begin{equation}\label{eq-ve-nu-not0}
 \xi_t = - M\xi \ .
\end{equation}
where $ M={-\partial_z^3+(6u_0+c)\partial_z  +6{u_0}_z
} = -\partial_z^3+12\wp(z)\partial_z  + 12\wp'(z) $ is the third-order differential operator introduced in \eqref{eq-operator-M}. We seek solutions of equation \eqref{eq-ve-nu-not0} in separable form, 
\begin{equation}
 \xi(z, t )=e^{-\nu t} \eta(z) \ .   
\end{equation}
This approach leads to the following differential equations
\begin{equation}\label{ec-var-no-not-0}
    M\ {\eta}=\nu \ {\eta} \quad , \quad \textrm{for }\nu\not= 0 \ ,
\end{equation}
and, for $\nu= 0$, also  we must solve the equation
\begin{equation}\label{def-VE-tangencial}
    M\ {\eta}=0 ,
\end{equation}
that provides {\it static solutions } since in this case  $\xi(z ,t)= \eta(z)$. We will study this case in section \ref{sec-nu-0}. We call equation  \eqref{def-VE-tangencial} \ {\it the tangential variational equation}

\medskip

Next, we will explain the method we follow to find a solution to equation \eqref{ec-var-no-not-0}. 

First, it is necessary to compute the centralizer  $\cZ (M)$ of the operator $M$. In the case of the cnoidal wave, this centralizer is nontrivial. It is generated, as a $\coC[M]$-module, by the following operators, computed by direct calculations using Maple,
\begin{align*}
           A_1 =&\partial^7 -28\wp(z)\partial^5 -84\wp' (z)\partial^4 -560\wp(z)\partial^3 -560\wp(z)\wp'(z)\partial^2 +\\ 
           &\left(-6720 (\wp(z))^3+1400 (\wp'(z))^2+2352 \wp(z) g_2 +1560 g_3 \right)\partial+\\
           &784\wp'(z) g_2 \\
           A_2 =&\partial^5-20\wp(z)\partial^3-40\wp'(z)\partial^2 +\left(
   -120 (\wp(z))^2 +18 g_2 
   \right)\partial +1 .
\end{align*}
Observe that $A_1 $ is an operator of order $7$ and $A_2$ has order $5$, as is required in Theorem \ref{thm-Gooddearl-forOrder3}. 
Moreover, we obtain the differential ring isomorphism
\begin{equation*}
    \cZ(M)  \simeq \coC[\widetilde{\Gamma}]=\frac{\coC [\lambda, \mu_1 ,\mu_2]}{\BC(M)} \ 
\end{equation*}
where $\BC(M) $ is the Burchnall-Chaundy ideal  gerenated by 
\begin{align*}
f_1  (\nu , \mu_1 )=&{\nu}^{7}+32\,{ g_3}\,{\nu}^{5}+ 256\left( -1088\,{{ g_2}}^{3}+
\,{{ g_3}}^{2} \right) {\nu}^{3}+25600\,{{ g_2}}^{2}\mu_{{1}}{\nu}
^{2}-\\ 
&784\,{ g_2 }\,{\mu_{{1}}}^{2}\nu+8\,{\mu_{{1}}}^{3}\\
f_2  (\nu , \mu_2 ) = & {\nu}^{5}+16\,{ g_3}\,{\nu}^{3} -4\,{ g_2}\left( \,\mu_{{2}}-1 \right) {\nu}^{2}+
  ( \,\mu_{{2}}-1  )^3 \\
f_3  (\mu_1 , \mu_2 )=&{\mu_{{2}}}^{7}-7\,{\mu_{{2}}}^{6}+ a_{5}{\mu_{{2}}}^{5}+ a_{4} {\mu_{{2}}}^{4}+a_3  {\mu_{{2}}}^{3}+  a_2  {
\mu_{{2}}}^{2} +a_1  \mu_{{2}} +a_0 \ ,
\end{align*} 
with $a_{5}= 2176\,{\ g_2 }\,{ g_3 }+21 $, 
 $a_{4}= -44608\,{ g_{2}}^{2}\mu_{{1}}-10880
\,{ g_2}\,{ g_3}-35 $, \break 
$a_3= -1287913472\,{
{ g_2 }}^{5}+1183744\,{ g_{2}}^{2}{ g_{3}}^{2}+178432\,{ g_{2}}
^{2}\mu_{{1}}+64\, g_{3}\,{\mu_{{1}}}^{2}+21760\,{ g_{2}}\,{ g_{3}}+
35 $, 
$a_2 =  3863740416\,{{ g_2 }}^{5}-29593600
\,{{ g_{2}}}^{3}{ g_{3}}\,\mu_{{1}}-3551232\,{{ g_{2}}}^{2}{{ g_{3}}}^
{2}-2688\,{ g_{2}}\,{\mu_{{1}}}^{3}-267648\,{{ g_{2}}}^{2}\mu_{{1}}-
192\,{ g_{3}}\,{\mu_{{1}}}^{2}-21760\,{ g_{2}}\,{ g_{3}}-21 $,
$a_1 =  5030912\,{{ g_{2}}}^{4}{\mu_{{1}}}^{2}-
3863740416\,{{ g_{2}}}^{5}+59187200\,{{ g_{2}}}^{3}{ g_{3}}\,\mu_{{1}}
+3551232\,{{ g_{2}}}^{2}{{ g_{3}}}^{2}+5376\,{ g_{2}}\,{\mu_{{1}}}^{3}
+178432\,{{ g_{2}}}^{2}\mu_{{1}}+192\,{ g_{3}}\,{\mu_{{1}}}^{2}+10880
\,{ g_{2}}\,{ g_{3}}+7 $, 
$a_0 = -5030912\,{{ g_{2}}}^{4}{\mu_{
{1}}}^{2}-591872\,{{ g_{2}}}^{2}{ g_{3}}\,{\mu_{{1}}}^{3}+1287913472\,
{{ g_{2}}}^{5}-29593600\,{{ g_{2}}}^{3}{ g_{3}}\,\mu_{{1}}-32\,{\mu_{{
1}}}^{5}-1183744\,{{ g_{2}}}^{2}{{ g_{3}}}^{2}-2688\,{ g_{2}}\,{\mu_{{
1}}}^{3}-44608\,{{ g_{2}}}^{2}\mu_{{1}}-64\,{ g_{3}}\,{\mu_{{1}}}^{2}-
2176\,{ g_{2}}\,{ g_{3}}-1$. Observe that $\deg_{\mu_{1}} (a_i )\leq 5-i$ for $i=0, \dots , 5$.

The spatial spectral curve $\tilde{\Gamma}$ associated to $M$ is defined by the polynomials $f_1 , f_2 , f_3 $. Moreover we can consider each $f_i$ as defining polynomial for a plane spectral curve. In this case, these curves $f_i = 0$ are elliptic curves, whose Weierstrass normal form can be derived so as to yield a parametric representation of the spectral curve. This can be accomplished using the Maple package, whose implemented algorithm is based on the work of M. van Hoeij \cite{van1995algorithm}.

\medskip

\noindent{\bf The intrinsic right factor}.
The methods designed in \cite{RZ2024} allow factoring $M-\nu$ (see \eqref{eq-globalFac} in Appendix \ref{sec-Order-3}). The right factor $\partial+\sigma_2$ whose solution is $\eta_2 (z)$ is given by
\[
\sigma_2 (z,s)={\frac {N(z,s)}{D(z,s)}}
\]
where

    \begin{align*}
     N(z,s)&=-{\nu}^{3}-4\,{\it\wp'} \left( z
 \right) {\nu}^{2}+\left(7104\,  {\it 
\wp} \left( z \right)   ^{3} -1072\,{\it\wp} \left( z \right) { g_2 }-720\,{ g_3 }\right)\nu+  \\
& \ \left(2\,{ g_2 }-24\,  {\it 
\wp} \left( z \right)   ^{2}\right)\mu_{{2}
}+  \\
&-4320\,{\it\wp} \left( z \right) {\it 
\wp'} \left( z \right) { g_2}+24\,
 \left( {\it\wp} \left( z \right) 
 \right) ^{2}-2\,-2880\,{\it \wp'} \left( z \right) { g_3}+
28800\, \left( {\it\wp} \left( z \right) 
 \right) ^{3}{\it\wp'} \left( z
 \right),\\
D(z,s)&=4\,{\it\wp} \left( z \right) {\nu}^{2}+\nu\,\mu_{{2}}+\left(16\,{\it\wp} \left( z \right) {\it\wp'} \left( z \right) -1\right)\nu+4\,{\it\wp'} \left( z \right) \mu_{{2}}-4\,{\it\wp'} \left( z\right) , 
\end{align*}
and $\nu=\nu (s)$, $\mu_2 = \mu_2 (s)$ is a parametrization of the curve $\{f_2 =0 \}$ in $\coC^2_{(\nu , \mu_2 )}$.

Consequently, we  obtain a global right-hand factor, which is compatible with $\partial+\sigma_1 $ and $\partial+\sigma_3$, and give rise to a global function $\eta$ over a Zariski open subset of the spectral curve -- a solution to the spectral problem~\eqref{ec-var-no-not-0}. This function is consequently parametrized by the points of the spectral curve associated with the operator~$M$. In particular,  we obtain in the $(\nu ,\mu_2 )$--plane, the first order differential equation
\[
D(z,s ) \partial \eta_2 + N(z,s) \eta_2 =0 \ .
\]
Observe that $\eta = \eta_2 (z,s ) =\eta_2 (z, P)$ for $P=(\nu , \mu_1 ,\mu_2 )\in\tilde{\Gamma}\setminus Z$ in $\bbC^3_{( \nu ,\mu_1 , \mu_2 )}$. Thus,  we obtain the formula
\[
\xi = e^{-\nu t }\eta (z, P)=e^{-\nu t }\eta (z, \nu ,\mu_1 , \mu_2 ) ,
\]
whenever $D(z,s )\not=0$ and $\nu\not=0$.

\subsection{Static Closed-Form Solutions}\label{sec-nu-0}

In this section we compute a basis of solutions for the tangential variational equation
\begin{equation}\label{ec-nu=0-for sol}
    M\ {\eta}=0 ,
\end{equation}
where $ M={-\partial_z^3+(6u_0+c)\partial_z  +6{u_0}_z
} = -\partial_z^3+12\wp(z)\partial_z  + 12\wp'(z) $. We will use the base of the solutions obtained in \ref{sec-A}.

Next, we will proceed to provide static solutions of the variational equation of the KdV equation around the cnoidal wave $u_{0} (z)=2\wp(z)-\frac{c}{6}$. From \eqref{eq:vesta}. we now deal with the problem of finding solutions to the equation
\begin{equation}
   M\xi = ( -\partial_z^3+12\wp(z)\partial_z  + 12\wp'(z) )\xi =0.
\end{equation}
Consequently, we face the study of the third-order operator $M$.
The adjoint variational equation is now
\begin{equation}\label{eq:avesta-cnoidal}
  M^\dag \psi  =(\partial_z^3+12\wp(z)\partial_z)\psi =0 .
\end{equation}
To solve equation \eqref{eq:avesta-cnoidal}, we will use the base of solutions obtained previously associated with the spectral problems of the operator Schr\"odinger and its spectral  curve $\Gamma$ defined in \eqref{eq:cnsp}.

First of all, let us recall that we have obtained the following solutions of the spectral problem associated with the points $P=(E,\mu)\in \Gamma$.

\begin{enumerate}
    \item At non branch points, we have obtained a basis of solution $\{ \phi_1 , \phi_2 \} $ of\break  $L-E=-\partial_z^2 +2\wp (z)-\wp (\imath s)$ with
    \begin{equation*}
        \varphi_i (z,s ) =\frac{{\bm \sigma} (z+(-1)^{i+1} s )}{{\bm \sigma} (s  ){\bm \sigma} (z)}e^{(-1)^{i} \, z \, {\bm\zeta} (s  )} \ ,   \ , \quad i =1, 2 ,
    \end{equation*}
    for $(z,s ) \in \left( \Omega\setminus \{ \omega_1 , \omega_2 ,\omega_3 \} \right)\times \left( \Omega\setminus \{ \omega_1 , \omega_2 ,\omega_3 \} \right)$.
    \item At the branch points $B_j = (-e_j , 0)$, $j=1, 2, 3$, we have obtained a basis of solution $\{  \widehat{\varphi_{j, e_j } } , \varphi_{j, e_j} \} $ of $L-E=-\partial_z^2 +2\wp (z)+ e_j $ with
    \begin{equation*}
         \widehat{\varphi_{j, e_j } }:=1/\sqrt{\wp(z)-{e_j }} \ , \quad \varphi_{j, e_j } = \sqrt{\wp(z)-{e_j }},
    \end{equation*}
    for $z\not= \omega_i $ in a fundamental domain $\Omega$
\end{enumerate}

\color{black}

In sections  \ref{sec-B} and \ref{sec-A}, we have obtained a closed form solutions basis of the adjoint variational equation $0 = M^\dag \psi$, since $u_0$ is an algebrogeometric potential. We detail the formulas that we have obtained in the case of the cnoidal wave.

\begin{enumerate}
    \item At non branch points of $\Gamma$, we have obtained a basis of solution $\{ \psi_1 , \psi_2 , \psi_3 \} $  with
    \begin{equation*}
       \psi_i = \varphi_i (z,s )^2  , \quad i =1, 2 , \quad \psi_3 = \varphi_1  \varphi_2  =\frac{{\bm \sigma} (z+ s ){\bm \sigma} (z- s )}{{\bm \sigma} (s  )^2 {\bm \sigma} (z)^2 } =\wp (s)-\wp (z)\ ,
    \end{equation*}
    for $(z,s ) \in \left( \Omega\setminus \{ \omega_1 , \omega_2 ,\omega_3 \} \right)\times \left( \Omega\setminus \{ \omega_1 , \omega_2 ,\omega_3 \} \right)$. 
    \item At the branch points $B_j = (-e_j , 0)$, $j=1, 2, 3$, we have obtained a basis of solution $\{ \psi_{1, e_j } , \psi_{2, e_j } , \psi_{3, e_j } \} $  with
    \begin{equation*}
        \psi_{1, e_j } (z)= ( \widehat{\varphi_{j, e_j } } )^2 \ , \  \psi_{2, e_j } (z) = (\varphi_{j, e_j } )^2  \  , \  \psi_{3, e_j }=1 \ ,
    \end{equation*}
    for $z\not= \omega_i $ in a fundamental domain $\Omega$
\end{enumerate}

Consequently, we obtain the following formulas for the corresponding (static) solutions of the variational equation $0=\xi_t = M \xi$.

\begin{enumerate}
    \item At non-branched points of $\Gamma$, we have obtained a basis of solution $\{ \widetilde{\psi_1 }, \widetilde{\psi_2 }, \widetilde{\psi_3 } \} $  with
    \begin{equation*}
       \widetilde{\psi_i } = \partial_z (\psi_i )  , \quad i =1, 2 , \quad \widetilde{\psi_3 }= \partial_z ( \psi_3 ) = -\wp' (z)\ ,
    \end{equation*}
    for $(z,s ) \in \left( \Omega\setminus \{ \omega_1 , \omega_2 ,\omega_3 \} \right)\times \left( \Omega\setminus \{ \omega_1 , \omega_2 ,\omega_3 \} \right)$.
    \item At the branch points $B_j = (-e_j , 0)$, $j=1, 2, 3$, we have obtained a solution basis $\{ \widetilde{\psi_{1, e_j }} , \widetilde{\psi_{2, e_j }}  \} $  with
    \begin{equation*}
       \widetilde{\psi_{1, e_j }} =\partial_z ( \psi_{1, e_j })= \frac{-\wp' (z)}{(\wp (z)-e_j )^2 } \ , \quad  \widetilde{\psi_{2, e_j } } = \partial_z ( \psi_{2, e_j })= -\wp'(z)  \  , \  
    \end{equation*}
    for $z\not= \omega_i $ in a fundamental domain $\Omega$. 
\end{enumerate}

Let us observe that equation \eqref{ec-nu=0-for sol} is the equation \eqref{ec-aux-Lame} corresponding to the second symmetric power of the algebrogeometric Schr\"odinger operator \eqref{eq-o2-conoidal} for $\Tilde{E}= E+\frac{c}{6}$. Therefore, we obtain the following closed form formulas for a basis of solutions of \eqref{ec-nu=0-for sol}:
\begin{equation}
    \eta_1 (z)= \wp (z)+\Tilde{e}_1 \ , 
     \eta_2 (z)= \wp (z)+\Tilde{e}_2 \ , 
      \eta_3 (z)= [ \eta_1 (z)  \eta_2 (z) ]^{1/2} .
\end{equation}
with $\Tilde{e}_i $ the zeros of the polynomial $4 X^3-g_2 X -g_3$ associated to the elliptic curve $\Gamma$ and the Weiestrass funcion $\wp (z) = \wp  (z; g_2 , g_3 )$.

\begin{rem}
 The hidden time evolution of these solutions can be studied by reversing the change of variable, and then considering the cnoidal wave at $(x,t)$. Consequently, we obtain:
    \begin{equation}
    \frac{\partial \eta_1 }{\partial t} (x,t)= -c\wp' (x-ct) \ , 
     \frac{\partial \eta_2 }{\partial t} (x,t)= -c\wp' (x-ct)\ , 
\end{equation}
 \begin{equation}
         \frac{\partial \eta_3}{\partial t} (x,t)= 
        -c\wp' (x-ct)  \frac{   (2\wp' (x-ct)-\Tilde{e}_3 )  }{2 \eta_3 (x,t) } 
\end{equation}   
\end{rem}

\bigskip

\section{Conclusions}

This work establishes a direct bridge between the variational equation of the
KdV equation around a cnoidal wave and the spectral--Galoisian framework of
algebro-geometric differential operators. Starting from the Lam\'e--Schr\"odinger
operator associated with the cnoidal background, we show that the solutions of
the variational and adjoint variational equations can be constructed
systematically from the geometry of the corresponding spectral curves.

The first original contribution of the paper is the use of the second
symmetric power of the Schr\"odinger operator as a unifying mechanism
connecting three apparently different objects: squared eigenfunctions,
conserved densities of the KdV hierarchy, and solutions of the variational
equations. In this way, the classical Gel'fand--Dikii recursion, the
Hamiltonian structure of KdV, and the differential-Galoisian viewpoint are
incorporated into a single framework.

A second contribution is the explicit construction of closed-form solutions
of the variational equation around the cnoidal wave. On the one hand,
solutions are obtained from the Hermite--Halphen eigenfunctions of the Lam\'e
equation and their squared products. On the other hand, a new family of
solutions is derived from the spectral analysis of the third-order operator
governing the variational equation itself, using its nontrivial centralizer
and the associated spatial spectral curve $\Gamma_1$. This provides explicit
formulas parametrized by points of the spectral curve and yields a spectral
interpretation of the variational dynamics.

A third contribution is the identification of a remarkable compatibility
between three independent constructions: the spectral Picard--Vessiot
approach for the Schr\"odinger operator, the squared-eigenfunction formalism,
and the factorization theory of third-order algebro-geometric operators. At
the branch points of the spectral curve $\Gamma_0$ these constructions
coalesce and produce the same distinguished solutions, revealing a special
geometric role for the ramification locus $Z$.

\medskip
\noindent\textbf{Towards a Galoisian reading of the ramification locus.}
The compatibility phenomenon established in Theorem B admits a natural
Galoisian reformulation. Recall that, in the theory of obstructions to
integrability initiated by Ziglin and developed in the differential-Galoisian
setting by Morales-Ruiz and Ramis, the linearization of a Hamiltonian system
around a particular solution --- the variational equation --- carries a
differential Galois group whose non-virtual-abelianity obstructs
integrability. In the present setting, the spectral curve $\Gamma_0$ plays the
role of a moduli space for the Picard--Vessiot extensions of the
Schr\"odinger operator (Section~2.1), and the differential Galois group of
the variational equation along the cnoidal background should therefore be
understood as a group scheme fibered over $\Gamma_0$, generically of maximal
type away from~$Z$.

Theorem~B shows that at the branch points $B_1, B_2, B_3 \in Z$ this fibered
structure undergoes a sharp collapse: the two Hermite--Halphen eigenfunctions
merge into a single Lam\'e solution, the squared-eigenfunction construction
degenerates into static solutions, and the intrinsic factorization of $M$
localizes onto the same distinguished family. We regard this as strong
structural evidence for the following:

\begin{conjecture}
The differential Galois group of the variational equation, viewed as a group
scheme over the spectral curve $\Gamma_0$, undergoes a proper degeneration
(a drop in dimension or in the type of its identity component) exactly over
the ramification locus~$Z$. 
\end{conjecture}

Establishing this conjecture rigorously requires computing the differential
Galois group of the variational equation directly --- rather than inferring
its behaviour from the compatibility of solution families --- and relating
this computation to the monodromy of the fibration $\Gamma_0 \to
\mathbb{P}^1$ near~$Z$. We expect the third-order operator~$M$ and its
centralizer $Z(M)$, whose Burchnall--Chaundy ideal was computed explicitly in
Section~4.2, to provide the right algebraic framework for this computation,
since the degeneration of $Z(M)$ near $Z$ should mirror the degeneration
described in Theorem~B at the level of solutions.

Beyond the case of cnoidal waves treated here, we conjecture that this
phenomenon is not an artifact of genus one: for higher finite-gap
potentials, the ra\-mification locus of the associated hyperelliptic spectral
curve should analogously mark the loci of Galoisian degeneration of the
variational equation along the corresponding higher members of the KdV
hierarchy. Establishing this connection in a fully Galoisian setting, and
extending the present spectral-geometric description of variational
dynamics beyond genus one, constitutes the natural direction for future
research.

\bigskip

\centerline{\rule{5cm}{0.4pt}}

\bigskip

\noindent {\sf Acknowledgments:} We kindly thank all members of the Integrability Madrid Seminar for many fruitful discussions:  R. Hern\'andez-Heredero, S. Jim\'enez, A. Jim\'enez-Pastor, A. P\' erez-Raposo, J. Rojo Montijano, S. Rueda and R. S\'anchez; the member in Colombia: D. Bl\' azquez-Sanz, and the member in Rep\'ublica Dominicana: P.B. Acosta-Hum\' anez.
\para

The first author is indebted to the Dynamical System Research Group
of the
Department of Applied Mathematics and Physics of Kyoto University and,
in
particular, to Shotaro Yamazoe And Kazuyuki Yagasaki for some
motivating
discussions when this author stays as a visiting professor  at
Kyoto
University on the autum of 2018.

\para

The first author is a member of the UPM Research Group ``Modelos ma\-tem\'aticos no li\-neales" and the SPECTINGAL group of the ETSE school,  UPM.  M. A. Zurro is partially supported by the grant PID2021-124473NB-I00, ``Algorithmic Differential Algebra and Integrability" (ADAI)  from the Spanish MICINN.

\appendixpage

\appendix


\section{On KdV differential polynomials and KdV ordinary differential operators}\label{sec-KdV-polinomials}

Let $K$ be a differential field with field of constants $\coC$, an algebraically closed field of zero characteristic. Consider a  differential indeterminate $u$ over $\coC$. We  call {\it formal Schr\"odinger operator} $L(u)=-\partial^2 +u$ with coefficients in the ring of differential polynomials 
$$\coC\{u\}=\coC[u,u',u'',\ldots],$$ 
where $u'$ stands for $\partial (u)$ and $u^{(n)}=\partial^n (u)$, $n\in\bbN$. In this section we will work with the formal  Schr\" odinger operator $L=L(u)$.

Let us consider the pseudo-differential operator
\begin{equation}\label{eq-recursion}
\cR=-\frac{1}{4}\partial^2+u+\frac{1}{2}u'\partial^{-1}\mbox{ and its adjoint }
\cR^*=-\frac{1}{4}\partial^2+u-\frac{1}{2}\partial^{-1}u'.
\end{equation}
Observe that $\cR^*=\partial^{-1}\cR\partial$. The operator $\cR^*$ is a recursion operator of the KdV equation (see \cite{olver1993applications}, p. 319). Applying the recursion operator $\cR$, 

we define:

\begin{equation}\label{eq-kdv}
\kdv_0:=u',\,\,\, \kdv_n:=\cR(\kdv_{n-1}),\mbox{ for }n\geq 1.
\end{equation}
Applying $\cR^*$ we define:
\begin{equation}\label{eq-fn}
v_0:=1,\,\,\, v_n:=\cR^*(v_{n-1}),\mbox{ for }n\geq 1.
\end{equation}

Hence for $n\in\bbN$ it holds
\begin{equation}\label{eq-vkdv}
2\partial(v_{n+1})=\kdv_n.
\end{equation}

As in \cite{GH, MRZ1}, we define a family of differential operators in $\coC\{u\}[\partial]$ of odd order (see also \cite{Dikii}, \cite{Novikov})
\begin{equation}\label{eq-A2s+1}
P_1:=\partial,\,\,\,P_{2n+1}:=v_{n}\partial-\frac{1}{2}\partial(v_{n})+P_{2n-1}L,\mbox{ for }n\geq 1.
\end{equation}

Observe that

\begin{equation}\label{eq-P2n+1}
P_{2n+1}=\sum_{l=0}^n\left(v_{n-l}\partial-\frac{1}{2}\partial(v_{n-l})\right)L^l.
\end{equation}
We will call the differential operators $P_{2n+1}(u)$ the {\it KdV differential operators}, \cite{MRZ1}. In \cite{MRZ1} is proved the following results.

\begin{lem}\label{lem-kdv}
The formulas for  $\kdv_n$ and $v_n$ give differential polynomials in $\coC\{u\}$.
\end{lem}
Next, we define the total derivative of a differential polynomial following the definition given in \cite{olver1993applications}. For $P=P(x,u,u' ,u'',\dots)\in \coC \{ u\}$ its total derivative is
\begin{equation}
    \D P = \dfrac{\partial P}{\partial x}
    +
    \sum_{k=0}^n u^{(k+1)} \dfrac{\partial P}{\partial u^{(k)}}
\end{equation}
where $\partial / \partial x $ is the given derivation in the ring of differential polynomials. Following \cite{Kolchin1973}, and with a slight abuse of notation, we write $\partial $ instead of $\D$  whenever the meaning is clear from the context.

\begin{lem}\label{lem-v}
The formula for  $\kdv_m v_n$ is a total derivative in $\coC\{u\}$ with $m, n$ non negative integers.
\end{lem}
\begin{proof} Since $\cR$ and $\cR^*$ are adjoint operators we have $p\cR (q)=q\cR^*(p)+\partial (a)$, $p,q,a\in C\{u\}$. Thus for $p=1$ and $q=\kdv_m$ we get
\[\cR^{n}(\kdv_m )=\kdv_m (\cR^*)^n (1)+\partial (a_{n,m}),\,\,\,\mbox{ for } a_{n,m}\in C\{u\}.\]
Then $ \kdv_{n+m} =\kdv_m  v_n +\partial (a_{n,m} )$ which implies that $\kdv_m  v_n  = \partial ( 2 v_{n+m+1} - a_n )$, hence it is a total derivative in $\coC\{u\}$.
\end{proof}

\begin{lem}\label{lem-fs}
For $n\in \bbN$ it holds $[P_{2n+1},L]=\kdv_n$.
\end{lem}

Let observe that the operators $P_{2n+1}$ have the important property that the commutator  $[P_{2n+1},L]$ is a differential operator in $\coC\{u\}[\partial]$ but Lemma \ref{lem-fs} shows that it has order zero, it is the multiplication operator by the $\kdv_n$ differential polynomial. This is the famous Lax representation of $\kdv_n$, see \cite{GH, Novikov}.

Now let us consider algebraic indeterminates $c_n$, $n\geq 1$ over $C$. We
define an extended family of {\it KdV differential polynomials} $\KdV_n(u,c^n)$, $n\in\bbN$ in the differential indeterminate $u$ and the list of algebraic indeterminates $c^n=(c_1,\ldots ,c_n)$.
\begin{equation}
\begin{array}{rl}\label{eq-KdV}
 \KdV_0 = \KdV_0 (u,c^0 )     & :=\kdv_0 = u', \\
  \KdV_n (u,c^n )   & :=\kdv_n+\sum_{l=0}^{n-1} c_{n-l} \kdv_l, \quad \mbox{ for } n\geq 1 .
\end{array}
\end{equation}
Then, we have
\begin{equation}\label{eq-KdV-R}
   \KdV_n (u,c^n ) = \cR (  \KdV_{n-1} (u,c^{n-1} )  ) +c_n \kdv_0 .
\end{equation}
We can construct an extended family of {\it KdV differential operators} whose coefficients are differential polynomials in $u$ and $c^n$,
\begin{equation}\label{eq-Ac2s+1}
\hat{P}_1:=\partial\mbox{ and }\hat{P}_{2n+1}:=P_{2n+1}+\sum_{l=0}^{n-1} c_{n-l}P_{2l+1}, \mbox{ for } n\geq 1.
\end{equation}
One can easily check that
\begin{equation*}\label{eq-PLKdV}
[\hat{P}_{2n+1},L]=\KdV_n=2\partial(f_{n+1}),
\end{equation*}
for
\begin{equation}\label{eq-fn}
f_0:=v_0=1\mbox{ and } f_n:=v_n+ \sum_{l=0}^{n-1} c_{n-l}v_l, \mbox{ for } n\geq 1.
\end{equation}
Accordingly to \eqref{eq-fn}, we get the following equlity
\begin{equation*}\label{eq-fn-R}
     f_{n+1} = \cR^* ( f_n )+ c_{n +1} f_0 .
\end{equation*}

In particular, for $n=1$ and the  algebraic indeterminate  $c_0$, we obtain the family

\[\begin{array}{rl}
     \KdV_1 =&  \kdv_1+ c_{0} \kdv_0 =-\frac{1}{4} u'''+\frac{3}{2} u u' + c_{0} u' \\
    \hat{P}_{3} = & P_{3}+ c_{0}P_{1}= -\partial^3+\frac{3}{2} u \partial +\frac{3}{4}u' + c_{0}\partial .
\end{array}
\]

The KdV hierarchy in dimension 1+1 is defined by the following family of formulas.
\begin{equation}\label{eq-kdv-1+1}
    u_{t_n } = \kdv_n \quad , \quad n\in \bbN.
\end{equation}
An equivalent presentation is obtained by using the operators $P_{2n+1}$ defined in \eqref{eq-A2s+1} :
\begin{equation}
L_t = [\partial_t , L]=    [{P}_{2n+1},L].
\end{equation}\label{eq-P-kdv-1+1}
where de variables are assumed to be $(x,t)$ and the derivations $\Delta= \{ \partial=\partial_x , \partial_t \}$ with relations
\[
[\partial, x]=1 \quad , \quad [\partial_t ,t] =1 \ .
\]

\bigskip

\centerline{\rule{5cm}{0.4pt}}

\bigskip
Now we introduce some definitions and notation. The ring of pseudodifferential operators with coefficients in a differential ring A. 
\begin{equation*}
    A(( \D^{-1} ))= \left\{ \sum_{i=-\infty}^n a_i \D^{i} \ | \ a_i \in A \ , \ n\in \bbZ
    \right\} ,
\end{equation*}
where multiplication is defined by the formula
\begin{equation*}
    \D^n a -a\D^n = \sum_{n=1}^{\infty}n\cdot (n-1) \cdots (n-i+1) \partial^{i} (a) \D^{n-i}.
\end{equation*}
Its subset $ A [ \D ]= \left\{ \sum_{i=0}^n a_i \D^{i} \ | \ a_i \in A \ , \ n\in \bbN    \right\}$ is the ring of differential operators with coefficients in $A$. 

\bigskip

For each $\ell \in \bbN$, $\ell\not=0$, we consider the $\Delta$ differential ideal generated by the differential polynomial $\kdv_\ell$. Let denoted by $I_\ell = [ \kdv_\ell ]$ and define the differencial ring
 \begin{equation*}
     R_\ell = R/ I_\ell \, ,
 \end{equation*}
and the quotient map $\theta_\ell : R\rightarrow  R_\ell $. Abusing notation, we denote by $\theta_\ell$ the extension of $\theta_\ell$ to the pseudodifferential operator rings.
\begin{equation*}
  \theta_\ell : R(( \D^{-1} )) \rightarrow  R_\ell (( \D^{-1} ))  
\end{equation*}
We define the ring homomorphism $\cR_\ell  :R_\ell (( \D^{-1} ))   \rightarrow R_{\ell+1 }(( \D^{-1} ))  $ as the action of the recursion operator $\cR$ on $R_\ell (( \D^{-1} )) $, since $ \kdv_{\ell+1}:=\cR(\kdv_{\ell})$. Consequently, we obtain the following commutative diagram:

\

\centerline{
\begin{tikzpicture}[
  every node/.style={minimum size=0.8cm},
  ->, >=Stealth, thick
]
  \node (A) at (0, 1.5) {$R(( \D^{-1} ))$};
  \node (B) at (3, 3)   {$R_\ell (( \D^{-1} ))$};
  \node (C) at (3, 0)   {$R_{\ell+1} (( \D^{-1} ))$};

  \draw (A) -- node[midway, above, draw=none] {$\theta_\ell$} (B);
  \draw (A) -- node[midway, below, draw=none] {$\theta_{\ell+1}$ } (C);
  \draw (B) -- node[midway, right, draw=none] {$\cR_\ell $} (C);
\end{tikzpicture}
}

\section{ On second order algebrogeometric  ordinary differential operators}\label{sec-algebro-L2}

\bigskip

Let us consider a differential operator $L$ with coefficients in a differential field $(\Sigma,\partial)$, whose field of constants $C$ is algebraically closed and of characteristic zero. There are several characterizations of algebro-geometric operator, see for instance \cite{We}. We  state next what we use as the base characterization of algebro-geometric operators for this work, the Burchnall and Chaundy Theorem \cite{BC}, {adapted from \cite{We}}.
We consider the nontrivial case of operators $L\not\in C[\partial]$.

\begin{thm}\label{thm-algegeom}
Let $L$ be an order $2$ differential operator in $\Sigma[\partial]\backslash C[\partial]$. The following are equivalent:
\begin{enumerate}
    \item $L$ is an algebro-geometric operator.
    \item There exists an operator $P$ in $\Sigma[\partial]$ of order $2n+1$, and a polynomial $f(\lambda,\mu)=\mu^2+R_{2n+1}(\lambda)$ in $C[\lambda,\mu]$, with $R_{2n+1}$ of degree $2n+1$, such that $f(L,P)=0$.
    \item There exists an operator $P$ in $\Sigma[\partial]$ of order $2n+1$, such that $[L,P]=0$.
\end{enumerate}
\end{thm}

In the case of second order differential operators we would like to highlight the structure of the centralizer $\cC(L)$ of $L$ in the ring of differential operators $\Sigma[\partial]$ stablished in \cite{MRZ2}. 

\begin{thm}\label{thm-agCen}
Let $L$ be a second order  differential operator in $\Sigma[\partial]$. The following are equivalent:
\begin{enumerate}
    \item $L$ is an algebro-geometric operator.
    \item The centralizer of $L$ is nontrivial, $\cC(L)\neq C[L]$. More precisely, there exists {a unique monic operator $A_{2s+1}$ of minimal order $2s+1$ such that $\cC(L)=C[L,A_{2s+1}]$ and $A^2_{2s+1}+R_{2s+1}(L)=0$, with $R_{2s+1}(\lambda)$ in $C[\lambda]$ of degree $2s+1$.}
\end{enumerate}
\end{thm}

Moreover, for the algebrogeometric operator Schr\"odinger $L_s = -\partial^2+u_s$, the following results are valid.
\begin{thm}
 Given $L_s = -\partial^2+u_s$, the following statements are equivalent.
\begin{enumerate}
    \item {\it $L_s$ is algebro-geometric.}
    \item {\it There exists a unique monic operator $A_{2s+1}$ of minimal order $2s+1$  such that $\cC(L_s)=C[L_s,A_{2s+1}]$ and $A^2_{2s+1}+R_{2s+1}(L_s)=0$, with $R_{2s+1}(\lambda)$ in $C[\lambda]$ of degree $2s+1$.}
    \item {\it $u_s$ is a $KdV$-potential of KdV level $s$ (i.e. it satisfies one of the $\KdV_s$ equations of the KdV-hierarchy, see Appendix \ref{sec-KdV-polinomials}).}
\end{enumerate}  
\end{thm}
Let $\Gamma_s$ be the plane algebraic curve defined by the polynomial $f(\lambda, \mu) = \mu^2 +R_{2s+1} (\lambda )$. The spectral  Picard-Vessiot field for $L_s = -\partial^2+u_s$ can be described as follows.

First, consider $K(\Gamma_s )$ the fraction field of the domain
\[
K[\lambda ,\mu ] /(f_s )
\]
with $K=\coC\{ u_s \}$. Then, it can be shown that  The field of constants of $K(\Gamma_s )$ is $\coC(\Gamma_s )$. Let $\Upsilon$ a fundamental solution of the intrinsic right factor $\partial+\sigma$ constructed as $L_s -\lambda = (-\partial-\sigma )(\partial+\sigma )$ in $K(\Gamma_s )$, \cite{MRZ1}. Moreover, we prove that $K(\Gamma_s)\langle \Upsilon_s\rangle$ is a transcendent Liouvillian extension  of $K(\Gamma_s)$, whose field of constants is the field of the curve $C(\Gamma_s)$, \cite{MRZ2}.

Next we include a basis of solutions obtained for the Lam\'e potential $u_1 = 2\wp(z)$ in \cite{MRZ2}.
\[
 \Upsilon_i (z)= \frac{{\bm \sigma}(z+(-1)^{i+1} s )}{{\bm \sigma}(s ){\bm \sigma}(z)}\; \cdot \;\exp\left((-1)^{i+1} \mu\;\bm\zeta(s)\right) \qquad , \qquad i=1,2.
\]
where $\bm \sigma$ and $\bm \xi$ are the Weierstrass sigma and Weierstrass zeta functions respectively, \cite{WW}.

\subsection{The Lam\'e solutions}\label{sec-A}  

In this section, we review  the main ideas of  Halphen (and previously from  Hermite), in the Jacobi form of the Lam\'e equation, see \cite{HALP, WW}, from the point of view of this work. 

Halphen   in \cite{HALP} considered the differential operator
\begin{equation*} 
(L-\Tilde{E} ) \phi = -\phi''+(n(n+1)\wp(z)+B)\phi,
\end{equation*}
Then, the  squared symmetrical power of the operator $L-\Tilde{E}$ defined in  \eqref{ecu-Lame} is \cite{BRONSTEIN1996}
\begin{equation}\label{eq-o2-conoidal}
(L-\Tilde{E} )^{\odot 2}     =  -\partial_{zzz} + (8\wp (z)-4 \Tilde{E} )\partial_z +4\wp'(z) . 
\end{equation}
Then, he proved that 
\begin{equation}\label{ec-aux-Lame}
 (L-\Tilde{E} )^{\odot 2} R(z)=  -R_{zzz} + (8\wp (z)-4 \Tilde{E} )R_z +4\wp'(z)  R=0 ,
\end{equation}
for $R(z)=\wp(z)-B=\wp(z)+\Tilde{E}=\wp(z)-\wp(a) $ for some  $a$, and $R(z)$ is the product of two solutions of \eqref{eq-o2-conoidal}, 
 $R(z)=\varphi_1\varphi_2$. Then, considering the quotient,
\begin{equation}\label{formula-wronskian}
   \frac{\tw(\varphi_1 ,\varphi_2 )}{\varphi_1 \varphi_2 } \ ,
\end{equation}
it is obtained that  $\Tilde{\mu}$, the Wronskian of these two solutions, $\tw (\varphi_1 ,\varphi_2 )$, must be a nonzero constant. Consequently, $\{ \varphi_1 ,\varphi_2 \}$ is a basis of solutions of \eqref{ecu-Lame}.

In \cite{HALP}, Deuxi\`eme partie, p. 469, Halphen studied the  Lam\'e equations
\begin{equation} \label{eq:hLame}
\varphi''=(2\wp(z, g_2 , g_3 )-\Tilde{e_i })\varphi ,
\end{equation}
where $\Tilde{e_i }$ are the roots of the polynomial $4X^3-g_2 X +g_3$ assumed to be distinct roots. Hence he were considering the brach points of the elliptic curve $\cE$, since $ \Tilde{e_i }=-e_i  $. It is obtained the  particular solution $\phi_i$, called {\it Lam\'e solution}, defined by
\begin{equation}\label{sol-L} 
    \varphi_i  (z) :=\sqrt{\wp(z, g_2 , g_3 ) +\Tilde{e_i }} \ , i = 1, 2, 3,
\end{equation}
because 
\begin{equation}
- \varphi_i''+(2\wp(z)-\Tilde{e_i } )\varphi_i   
=
\,{\frac {-4\,\Tilde{e_i }^{3}+{ g_2 }\,\Tilde{e_i }-{ g_3 }}{4 \left( 
\wp \left( x,{ g_2 },{ g_3 } \right) +\Tilde{e_i } \right) 
^{3/2}}} =0 .
\end{equation}
Moreover, the determinant of the Wronskian matrix of $\phi_1 , \phi_2$ is 
\begin{equation*}
   \tw ( \varphi_1 , \varphi_2 ) = 
   \,\frac {{ \wp}' \left( x,{ g_2 },{ g_3 }\right)  \left( \Tilde{e_1 } -\Tilde{e_2 } \right) }{2\varphi_1 \varphi_2 }
\not=0 .
\end{equation*}
As a consequence of the Weiestrass function  formulas, see \eqref{elliptic_homogeneity}, we obtain that $\tilde{e_i } = -e_i =- \wp (\omega_i , g_2 ,g_3 ) = \wp (\imath \cdot  \omega_i , \  g_2 , \  -g_3 )$, with $\omega_i$ the half periods of the initial Weiertrass function. Thus, the half periods are roots of the solutions, since
\begin{equation*}
    \varphi_i (\omega_i ) = \sqrt{e_i + \tilde{e_i } }=0 .
\end{equation*}

Lam\'e's solutions \eqref{sol-L} are elliptic functions with periods double of the periods of the initial $ \wp (z) $, \cite{HALP, POO}, and  they are in a quadratic extension of the field of coefficients, $ \mathbb{C} (\wp (z), \wp '(z)) $. 

\begin{rem}
 The Lam\'e solutions are the eigenfunctions of the simple spectrum, $ e_i , \  i=1,2,3 $, of the Hill problem defined by the Lam\'e equation,  see\cite{WW}. In fact,  in order to have a well-defined Hill spectral problem,  it is convenient to consider $\wp (z+i\gamma),\, \gamma>0$ fixed (for instance, $\gamma=|\omega_2|$, being $i2\omega_2$ the imaginary period), to avoid the pole of the $\wp $ function;  this gives  again a cnoidal traveling wave solution.
   
\end{rem}

\section{Third-Order Ordinary Differential Operators Revisited}\label{sec-Order-3}
We begin by recalling some preliminary results on the spectral analysis of a  third-order operator $N$, which are treated in greater depth in~\cite{RZ2024}. The problem of finding a closed-form solution for $Ny =\eta \; y$ was addressed in  \cite{RZ2024}. Next we summarize several relevant results concerning this outcome for ease of reference. 

\medskip

Let $\Sigma$ be an ordinary differential field, $\coC$ its field of constants, assumed algebraically closed of characteristic zero. Let $M$ be a third-order ordinary differential operator  in $\Sigma [\partial ]$ We denote by $\cZ ((M)) $ the centralizer of $M$ in the skew field of pseudodifferential operators $\Sigma ((\partial^{-1}))$. The {Generalized Schur's Theorem} (Goodearl, 1983, \cite{Good}) establishes the following equality.

\[\cZ ((M))=\left\{\sum_{j=-\infty}^m c_j Q^j\mid c_j\in \coC, m\in\bbZ\right\}, \ \textrm{with } Q=M^{1/3}.\]

This ring  $\cZ ((M))$ is commutative, and the centralizer  $\cZ (M)$ of $M$ in the ring of differential operators $\Sigma [\partial]$ is    
\[\cZ(M)=\cZ ((M))\cap \Sigma[\partial] .
\]
Moreover $\cZ(M)$ is a commutative differential domain whose field of fractions has transcendence degree $1$ over $\coC$. Consequently, {$Spec(\cZ(M))$} is an abstract {algebraic curve $\Tilde{\Gamma}$}, called {\it the spectral curve of $M$}. Deriving the algebraic equations that define this curve constitutes an intriguing problem in its own right. For the reader's convenience, we have included the relevant results from work \cite{RZ2024} upon which the present article draws.

\begin{thm}\label{thm-Gooddearl-forOrder3} (\cite{RZ2024})
 Given $M\in \Sigma[\partial]\backslash \coC[\partial]$, with $\cZ(M)\neq \coC[M]$, then  $\cZ(M)$ is a free $\coC [M]$-module of rank $3$. Moreover, there exists a basis $\{1,A_1,A_2\}$  of $\cZ(M)$ as a $\coC [M]$-module, where each $A_i$ is a monic operator in $\cZ(M)\backslash \coC[M]$ of minimal order 
$o_i:=\ord(A_i)\equiv i\ (mod\ 3 )$, $  i=1, 2$.
\end{thm}

Therefore we obtain the following decomposition of $\cZ(M) $ as a $\coC [M]$-module 
\begin{equation}\label{eq-cen_mod}
\cZ(M)    =\coC[M]\oplus \coC[M]A_1\oplus \coC[M]A_2=\coC[M,A_1,A_2] 
\end{equation}
In recent years, ordinary differential operator centralizers have attracted special attention, particularly from a computational point of view. See, for example, the recent work of S. Rueda and A. Jiménez-Pastor \cite{JPR2026-BCideals, rueda2026classification}.

The centralizer $\cZ(M) $ can be presented as a quotient ring. In fact, the kernel of the linear map\;  $\eM: \coC [\lambda,\mu_1, \mu_2] \rightarrow \Sigma[\partial]$ defined by 
\[
\eM(\lambda)=\textcolor{blue}{M},\,\,\,  \eM(\mu_1)=\textcolor{blue}{A_1 },\,\,\, \eM(\mu_2)=\textcolor{blue}{A_2 },
\]
is the Burchnall-Chaundy ideal of $M$:
\begin{equation*}\label{def-BCL}
{\BC (M):=\Ker(\eM)=\{g\in\coC [\lambda,\mu_1, \mu_2]\mid g(M,A_1,A_2)=0\}.    }
\end{equation*}
Moreover, from \cite{RZ2024} we obtain
\begin{equation*}\label{eq-isomorfismo}
    \cZ(M)  \simeq \coC[\widetilde{\Gamma}]=\frac{\coC [\lambda, \mu_1 ,\mu_2]}{\BC(M)} \ .
\end{equation*}
The (abstract) algebraic curve $\widetilde{\Gamma} =\Spec (\cZ(M))$ is referred to as the {\it spectral curve} of $M$. This curve — which is, in general, non-planar — admits a natural description as the zero locus of three polynomials. By computing differential resultants, we obtain the system
    $$\left\{\begin{array}{l}
     f_i=\dres(M-\lambda,A_i-\mu_i),\,\,\,i=1,2,\\
     f_3^r=\dres(A_1-\mu_1,A_2-\mu_2),
\end{array}\right.$$
 furthermore, establish that  $\BC(M)=(f_1,f_2,f_3)$ is a prime ideal. This ideal can be calculated explicitly, as well as a Gröbner basis of it for a monomial order adapted to the operator $M$, \cite{JPR2026-BCideals}.

\bigskip

\noindent{\bf The intrinsic right factor}

\medskip
In the following we present the results included in the work \cite{RZ2024} to make the treatment of the factorization of a third-order algebrogeometric operator more accessible to the reader.

In  \cite{RZ2024}, it is proved that the differential ideal generated by $\BC (M)$ and denoted $[\BC(M)]$ is a prime differential ideal of $\Sigma [\lambda,\mu_1,\mu_2]$. Then we have a differential domain
\begin{equation*}\label{eq-anillocurva}
{\Sigma[\Gamma]=\frac{\Sigma [\lambda,\mu_1,\mu_2]}{[\BC(M)]}}.
\end{equation*}
Its fraction field 
{$\Sigma(\Gamma)$} is a differential field with the extended derivation.  {The greatest common right divisor $\gcrd(M-\lambda,A_1-\mu_1,A_2-\mu_2)$ in $\Sigma(\Gamma)[\partial]$} is a first order monic operator $\partial+\sigma$, and
\begin{equation}\label{eq-globalFac} 
 \gcrd(M-\lambda,A_1-\mu_1,A_2-\mu_2) = \partial+\sigma
\end{equation}
equals
${\gcrd(M-\nu,A_1-\mu_1)=\gcrd(M-\nu,A_2-\mu_2)}$ 
 divides $\gcrd(A_1-\mu_1, A_2-\mu_2))$. Moreover, assume $M=\partial^3+u_1\partial+u_0$, then we obtain
{\begin{equation*}\label{eq-intrinsic-rfactor}
    M-\nu = \left(\partial^2 -\sigma\partial+
    u_1 -2\sigma' 
    +\sigma^2 
    \right)\cdot (\partial +\sigma ) ,
\end{equation*}
in  $\Sigma(\Gamma)[\partial ]$, under the condition
. 
We call the right factor $\partial +\sigma  $ of $M-\nu$ its {\sl intrinsic right factor over $\Gamma$}.

\bibliographystyle{plain}
\bibliography{Bibliography}



\end{document}